\documentclass[onefignum,onetabnum]{siamart171218} 

\newtheorem{Theorem}{Theorem}

\newtheorem{Lemma}{Lemma}
\newtheorem{Remark}{Remark}

\newtheorem{Assumption}{Assumption}
\newcommand{\ve}{\varepsilon}
\newcommand{\mb}{\mathbb}
\newcommand{\mc}{\mathcal}

\usepackage{lipsum}
\usepackage{amsfonts}
\usepackage{graphicx}
\usepackage{epstopdf}
\usepackage{algorithmic}
\ifpdf
  \DeclareGraphicsExtensions{.eps,.pdf,.png,.jpg}
\else
  \DeclareGraphicsExtensions{.eps}
\fi

\newsiamremark{remark}{Remark}
\newsiamremark{hypothesis}{Hypothesis}
\crefname{hypothesis}{Hypothesis}{Hypotheses}
\newsiamthm{claim}{Claim}

\headers{Convergence of the harmonic balance method}{Andrew Steyer and Robert J. Kuether}

\title{Convergence of the harmonic balance method for smooth Hilbert space valued differential-algebraic equations\thanks{Submitted to the editors DATE.
\funding{This paper describes objective technical results and analysis. Any subjective views or opinions that might be expressed in the paper do not necessarily represent the views of the U.S. Department of Energy or the United States Government. Supported by the Laboratory Directed Research and Development program at Sandia National Laboratories, a multimission laboratory managed and operated by National Technology and Engineering Solutions of Sandia LLC, a wholly owned subsidiary of Honeywell International Inc. for the U.S. Department of Energy’s National Nuclear Security Administration under contract DE-NA0003525.  SAND2021-14706 O}}}

\author{Andrew Steyer\thanks{Center for Computing Research, Sandia National Laboratories, Albuquerque, NM 87185
  (\email{asteyer@sandia.gov}).}
\and Robert J. Kuether\thanks{Engineering Sciences Center, Sandia National Laboratories, Albuquerque, NM 87185 (\email{rjkueth@sandia.gov}).}}

\usepackage{amsopn}

\makeatletter
\newcommand*{\addFileDependency}[1]{
  \typeout{(#1)}
  \@addtofilelist{#1}
  \IfFileExists{#1}{}{\typeout{No file #1.}}
}
\makeatother

\newcommand*{\myexternaldocument}[1]{%
    \externaldocument{#1}%
    \addFileDependency{#1.tex}%
    \addFileDependency{#1.aux}%
}

\ifpdf
\hypersetup{
  pdftitle={Convergence of the harmonic balance method for smooth Hilbert space valued differential-algebraic equations},
  pdfauthor={Andrew Steyer and Robert J. Kuether}
}
\fi

\myexternaldocument{ex_supplement}

\begin{document}

\maketitle

\begin{abstract}
We analyze the convergence of the harmonic balance method for computing isolated periodic solutions of a large class of continuously differentiable Hilbert space valued differential-algebraic equations (DAEs).  We establish asymptotic convergence estimates for (i) the approximate periodic solution in terms of the number of approximated harmonics and (ii) the inexact Newton method used to compute the approximate Fourier coefficients.  The convergence estimates are determined by the rate of convergence of the Fourier series of the exact solution and the structure of the DAE.  Both the case that the period is known and unknown are analyzed, where in the latter case we require enforcing an appropriately defined phase condition.  The theoretical results are illustrated with several numerical experiments from circuit modeling and structural dynamics.
\end{abstract}



\section{Introduction}

Periodic motion arises across a diverse range of mathematical models in engineering and the physical, biological, and social sciences.  Rigorous analysis of numerical methods for approximating periodic solutions to differential equations and, more generally, differential-algebraic equations (DAEs) is therefore essential in computer modeling and simulation.  Harmonic balance (HB) is a Fourier-Galerkin method that approximates periodic solutions to DAEs by approximating their expansion as a Fourier series.  In this paper we analyze the convergence of the HB method for a large class of continuously differentiable Hilbert space valued DAEs where the period of the solution is either known or unknown.

The main contributions of this paper are  Theorems \ref{thm:known}-\ref{thm:unknown} which are stated and proved in Sections \ref{sec:known}-\ref{sec:unknown}.  Theorem \ref{thm:known} gives convergence estimates for the HB method when the period of the exact solution is known, while Theorem \ref{thm:unknown} gives convergence estimates when the period is an unknown.  These theorems cover two aspects of convergence of the HB method: (i) the asymptotic convergence rate in the $L^2$-norm of the approximate to the exact periodic solution as a function of the number of approximated harmonics; (ii) convergence of an inexact Newton iteration to compute the Fourier coefficients of the HB approximation. These asymptotic convergence rates are determined by the convergence rate of the Fourier series of the exact periodic solution and by constants related to the structure of the DAE.  We demonstrate these theoretical results in Section \ref{sec:examples} on several examples arising in circuit modeling and structural dynamics.
 
Theorems \ref{thm:known}-\ref{thm:unknown} are proved by applying an inexact Newton-Kantorovich Theorem (Theorem  6.3 of \cite{FerreiraSvaiter2012}) using several fundamental estimates that we establish in Section \ref{sec:sharedestimates}.  The main hypotheses we employ are: (i) the DAE satisfies certain Lipschitz estimates and is also continuously differentiable in the state variables with the partial derivatives satisfying an additional weak differentiablity hypothesis (see Assumptions 
\ref{asp:Fasp}-\ref{asp:KPUP}) and (ii) the exact periodic solution satisfies an isolatedness condition (see Assumptions \ref{asp:known}-\ref{asp:unknown}).  When the period is unknown we enforce a phase condition that satisfies some additional mild hypotheses (see Section \ref{sec:unknown}).  Convergence with respect to other norms (e.g. sup-norm) is not analyzed since $L^2$ is the natural space for Fourier analysis and modifying the results herein for alternative norms is mostly a technical exercise.  We leave for future studies the relaxation of the smoothness assumptions (including implications related to solvers) and stability considerations.  

The theory we develop is a major advance in three ways: (i) the convergence analysis explicitly includes an inexact Newton procedure used for the nonlinear solves; (ii) results are formulated in terms of Hilbert space valued DAEs; (iii) both cases that the period is known and unknown are analyzed within the same framework.  Analysis of inexact versus exact Newton solvers is motivated by their use in applications (e.g. \cite{BlahosEtAl2020,DongPeng2009,RizzoliEtAl1997}) and necessitated by the fact that the HB method requires evaluating a residual defined by Fourier integrals.  The Hilbert space formulation makes our analysis applicable in a variety of contexts where HB is used: integral equations (e.g. \cite{KelleyMukundan1993}), partial differential equations (PDEs) (e.g. \cite{Huang2016}), as well as traditional finite-dimensional ODE and DAE problems.  Analysis of the HB method is relatively straightforward to generalize to abstract Hilbert space value problems since its structure only requires evaluation of a residual function.  The case where the period is unknown is important, especially for autonomous equations in a continuation framework (see e.g. \cite{Moore2005}).  Our results provide a unified foundation for further analysis of how the HB method is affected by spatial discretization error, solver considerations, and approximation errors.  

The HB method has been deployed in a wide range of applications, including sensitivity analysis \cite{sierraetal2021}, structural dynamics \cite{KrackGross}, electromagnetics and power systems \cite{LuZhaoYamada}, and fluid dynamics \cite{HallEtAl2013}. To the best of our knowledge, the earliest formal convergence results for the HB method are found in a paper due to Urabe \cite{Urabe1965}.  The results in \cite{Urabe1965} establish, for periodic solutions of finite-dimensional first-order ODEs where the period is known, the convergence of the HB solution to the exact solution (in terms of number of harmonics) as well as convergence of an exact Newton iteration to approximate the Fourier coefficients of the HB approximation.  However, the results in \cite{Urabe1965} do not establish the rate of convergence of the approximated solution to the true solution.  Additional convergence results found in the literature include extensions of the results of Urabe to autonomous ODEs where the period is an unknown \cite{Shinohara1981,Stokes1972} and analysis of the HB method for a class of integral equations \cite{KelleyMukundan1993}.  There are also a number of papers (e.g. \cite{Garcia-SaldanaGasull2013,kogelbreunung2021,Urabe1965}) using techniques similar to ours that are concerned with proving backward error type results i.e. proving the existence of a true periodic solution to a ``nearby'' problem when the error of the HB method converges. 

Recently \cite{WangEtAl2021}, a convergence theory of the HB method was developed for a class of finite dimensional second-order DAEs modeling structural mechanical systems.  We build upon these results in several ways: (i) we prove convergence of a nonlinear solver to the HB solution; (ii) we formulate isolatedness conditions that guarantee coercivity of the linearized HB residual rather than making nonlinear coercivity an assumption; (iii) we consider general Hilbert space valued DAEs; (iv) we consider the case where the period is an unknown.  While the results in \cite{WangEtAl2021} allow for discontinuities in the equations, we anticipate generalizing our analysis to this case in future work.  

The remainder of this paper is organized as follows.  We define notation and discuss some preliminaries in Section \ref{sec:preliminaries}.  In Section \ref{sec:convergence} we state and prove some necessary estimates and lemmas (Section \ref{sec:sharedestimates}), formulate the isolatedness conditions (Assumptions \ref{asp:known}-\ref{asp:unknown} in Sections \ref{sec:known}-\ref{sec:unknown}), and prove the main convergence results (Theorems \ref{thm:known}-\ref{thm:unknown} in Sections \ref{sec:known}-\ref{sec:unknown})).  The theoretical results are demonstrated on several examples in Section \ref{sec:examples}.  We conclude the paper with some final remarks in Section \ref{sec:conclusion}.

\section{Preliminaries}\label{sec:preliminaries}

\subsection{Notation and conventions}\label{sec:notation}

We discuss some notation and conventions used for the remainder of the paper.  Fix a Hilbert space $X$ and denote its inner-product and the associated induced norm by $\left<\cdot , \cdot \right>$ and $\|\cdot\|$, respectively.  We use the following notation and conventions where $Y$ and $Z$ denote arbitrary Banach spaces with norms $\|\cdot\|_Y$ and $\|\cdot\|_Z$, respectively: 
\begin{itemize}
    \item $\mathbb{N} = \{1,2,3,4,5,\hdots\}$ and  $\mathbb{N}_0 = \{0,1,2,3,4,5,\hdots\}$.
    
    
    \item $Y^j := \underbrace{Y \times \hdots \times Y}_{j-\text{times}}$ where $j \in \mb{N}$.
    
    \item For $j \in \mb{N}$, we denote the inner product and associated induced norm on the Hilbert spaces $X^j$ and $X^j \times \mb{R}$ using the same symbols $\left<\cdot,\cdot\right>$ and $\|\cdot\|$ that represent the inner product and induced norm on $X$:
    \begin{equation*}
        \begin{array}{cc}
        \left<x,y\right> = \sum_{n=1}^{j}\left<x_n,y_n\right>, &         \left<(x,\tau),(y,s)\right> = \sum_{n=1}^{j}\left<x_n,y_n\right> + \tau s, \\
          \|x\| = \left(\sum_{n=1}^{j}\|x_n\|^2\right)^{1/2}, &  \|(x,\tau)\| = \left(\sum_{n=1}^{j}\|x_n\|^2 + |\tau|^2\right)^{1/2} 
        \end{array}
    \end{equation*}
    where $\tau,s \in \mb{R}$ and $x,y \in X^j$ with $x= (x_1,\hdots,x_j)$ and $y = (y_1,\hdots,y_j)$.
    
    \item For each $j \in \mb{N}_0$ we denote the closed balls of radius $\xi \geq 0$ centered at $x \in X^j$ and $(x,\tau) \in X^j \times \mb{R}$ by:
$$\begin{array}{lcr}
\mc{B}(x,\xi) =\{y \in X^j : \|y-x\| \leq \xi\},\\[3pt]
\mc{B}((x,\tau),\xi) =\{(y,s) \in X^j \times \mb{R} : \|(y,s)-(x,\tau)\| \leq \xi\}.
\end{array}$$

    \item Let $U \subseteq \mb{R}$ be an open set and $j \in \mb{N}$.  We let $v^{(j)}$ denote the $j^{th}$-order derivative of a $j$-times differentiable function $v: U \rightarrow X$ and define $v^{(0)} = v$.  We let $d^j v$ denote the $j^{th}$-order weak derivative of a $j$-times weakly differentiable function $v:U \rightarrow X$ and define $d^0 v = v$.
    
    \item We use the following function space notation where $U \subseteq Y$ is open and $j \in \mb{N}$:
    
    \begin{itemize}
    
      \item $C^0(U,Z) = \{f :U \rightarrow Z: f \text{ is continuous on $U$}\}$,
    
     \item $C^j(U,Z) = \{f :U \rightarrow Z: f^{(j)} \text{ exists and is continuous on $U$}\}$,
     
    \item $C^{\infty}(U,Z) = \{f :U \rightarrow Z: f \in C^l(U,Z) \text{ for all } l \in \mb{N}\}$, 
     
     
     \item $L^2 = \{ f:\mb{R} \rightarrow X : f(t) = f(t+1) \text{ a.e. and } \int_0^1 \|f(t)\|^2 dt < \infty\}$,
     
     \item $H^0 = L^2$,
     
     \item $H^j =  
        \{f \in L^2 : d^l f \text{ exists on (0,1)}, \int_0^1 \|d^l f(t)\|^2dt < \infty, \forall 0 \leq l \leq j\},$
     
     \item $C^0 = C^0(\mb{R},X)$, $C^j = C^j(\mb{R},X)$, $C^{\infty} = C^{\infty}(\mb{R},X)$, $\tilde{C}^j = C^j \cap H^{j+1}$.
     
    \end{itemize}
    
   Norms on the spaces $L^2$ and $H^j$ where $j \in \mb{N}_0$ are defined by:
   
   \begin{itemize}
      \item $\|f\|_{L^2} = \left( \int_0^1 \|f(t)\|^2 dt\right)^{1/2}$ for all $f \in L^2$.
       
      \item $\|f\|_{H^j} = \left(\sum_{l=0}^{j} \|d^l f\|_{L^2}^2\right)^{1/2}$ for all $f \in H^j$.
   \end{itemize}
    
    \item $\mathcal{L}(Y,Z)$ denotes the Banach space of all bounded linear maps $Y\rightarrow Z$ with norm $\|\cdot\|_{\mc{L}(Y,Z)}$ defined by the operator norm induced by $\|\cdot\|_Y$ and $\|\cdot\|_Z$. We let $\mc{L}(Y) = \mc{L}(Y,Y)$.  For $j \in \mb{N}$ we use the following abbreviated notation when the context is clear:
    $$\|\cdot\|_{\mc{L}(X^j)} = \|\cdot\|, \quad \|\cdot\|_{\mc{L}(X^j \times \mb{R})} = \|\cdot\|, \quad \|\cdot\|_{\mc{L}(\mb{R},X)} = \|\cdot\|,$$
    $$\|\cdot\|_{\mc{L}(H^j)} = \|\cdot\|_{H^j}, \quad \|\cdot\|_{\mc{L}(L^2)} = \|\cdot\|_{L^2}.$$

\end{itemize}
 
\subsection{Fourier series of $X$ valued functions} 

If $v \in L^2$, then $v$ has a representation as a convergent Fourier series (see e.g. \cite{ArendtBu,DymMckean}):
\begin{equation}\label{eq:genericfourierexpansion}
\left\{
\begin{array}{lcr}
v(t) = v_{0} + \sum_{j=1}^{\infty}\left(\sqrt{2}v_{2j}\cos(2\pi j t) + \sqrt{2}v_{2(j-1)+1} \sin(2\pi jt) \right),\\[3pt]
v_{0} = \int_0^1 v(t) dt \in X,  \quad v_{2j} =  \int_0^1 v(t) \cos(2\pi jt) dt \in X, \quad j \in \mb{N}, \\[3pt]
v_{2(j-1)+1} = \int_0^1 v(t) \sin(2\pi jt) dt \in X, \quad j \in \mb{N}.
\end{array}\right.
\end{equation}
We will always use the real as opposed to complex form of the Fourier series.  The coefficients of $\sqrt{2}$ in the summation in \eqref{eq:genericfourierexpansion} are chosen to simplify the estimates in Lemma \ref{lem:Qlem1} in Section \ref{sec:preliminaries}.  Note that the integrals in \eqref{eq:genericfourierexpansion} are $X$ valued so-called Bochner integrals.  For $N \in \mb{N}_0$ we let $P_{N}:L^2 \rightarrow C^{\infty}$ denote the projection operator of Fourier series onto their first $N$ harmonics: with $v$ as in \eqref{eq:genericfourierexpansion} we define:
$$P_0(v) = v_{0}, \quad P_{N}(v) = v_{0} + \sum_{j=1}^{N}\left(\sqrt{2}v_{2j}\cos(2\pi j t) + \sqrt{2} v_{2(j-1)+1} \sin(2\pi jt) \right), \quad N \in \mb{N}.$$
The following bound is a consequence of Bessel's inequality:
\begin{equation}\label{eq:besselbound}
\|P_N\|_{L^2} \leq 1.
\end{equation}
For each $N \in \mathbb{N}_0$ we define $Q_{N}:X^{2N+1} \rightarrow C^{\infty}$ by:
\begin{equation}\label{eq:QN}
Q_{N}(x)(t) = x_{0} + \sum_{j=1}^{N}\left(\sqrt{2} x_{2j} \cos(2\pi j t) + \sqrt{2} x_{2(j-1)+1} \sin(2\pi j t) \right)
\end{equation}
where $x = (x_0,x_1,\hdots,x_{2N}) \in X^{2N+1}$.  We remark that $Q_N$ is always invertible since two trigonometric series are equal if and only if their coefficients are the same.  For $j \in \mb{N}_0$, denote by $Q_{N}^{(j)}(x)$ the $j^{th}$-order derivative of $Q_{N}(x) = Q_{N}(x)(t)$ with respect to $t$, which is well-defined everywhere since $Q_N(x) \in C^{\infty}$ for all $x \in X^{2N+1}$.  Denote  by $DQ_{N}^{(j)}$ the derivative of $Q_{N}^{(j)} = Q_{N}^{(j)}(x)$ with respect to $x$, whenever it exists.  We shall show (Lemma \ref{lem:Qlem1}) that $D Q_N^{(j)}(x)$ exists for all $x \in X^{2N+1}$ and $j \in \mb{N}_0$.

\subsection{An $X$ valued DAE}\label{sec:dae}
Fix $k \in \mb{N}$ and $G:X^{k+1} \times \mb{R} \rightarrow X$ for the remainder of the paper.  Consider the following DAE:
\begin{equation}\label{eq:dae1}
G(u , u^{(1)},\hdots,u^{(k)} , t) = 0, \quad u = u(t) \in X.
\end{equation}
For $j \in \{1,\hdots,k+1\}$ we use the symbol $\partial_j G$ to denote the partial derivative of $G$ with respect to its $j^{th}$ argument, whenever this derivative exists.  Assuming $\partial_j G$ is well-defined, we let $\partial_j G(x_0,\hdots,x_k,t_0)$ denote $\partial_j G$ evaluated at $(x_0,\hdots,x_k,t_0) \in X^{k+1} \times \mb{R}$.  We will leave the DAE \eqref{eq:dae1} in higher-order form since: (i) conversion to first-order form may not result in equivalent algorithmic formulations and (ii) it may sometimes be more computationally efficient to leave \eqref{eq:dae2} in higher order form than to reduce the order by increasing the dimension.  

If $\tau > 0$, then $q_1 \in C^k(\mathbb{R},X)$ is a $\tau$-periodic solution to \eqref{eq:dae1} if and only if $q_2 \in C^k(\mb{R},X)$ defined by $q_2(t):= q_1 (\tau t)$ is a $1$-periodic solution to the rescaled DAE:
\begin{equation}\label{eq:dae2}
G(u,u^{(1)} \tau^{-1},\hdots,  u^{(k)} \tau^{-k}, \tau t) = 0.
\end{equation}
Motivated by this, we follow the traditional approach of approximating a $\tau$-periodic solution $q_1$ to \eqref{eq:dae1} by approximating a $1$-periodic solution $q_2$ to \eqref{eq:dae2}.  Given $\tau > 0$ and $q \in \tilde{C}^k$ we let $F(q,\tau)$ denote the map defined by:
$$t \mapsto G( q(t),q^{(1)}(t) \tau^{-1},\hdots, q^{(k)}(t) \tau^{-k},\tau t)$$ 
with $F(q(t),\tau)$ denoting the value of $F(q,\tau)$ at $t \in \mb{R}$.   If $\partial_j G$ is well-defined, then for $j=1,\hdots,k+1$ we denote by $\partial_j F(q,\tau)$ the map defined by:
$$t \mapsto \partial_j G(q(t),q^{(1)}(t)\tau^{-1},\hdots,q^{(k)}(t)\tau^{-k},\tau t)$$
 and let $\partial_j F(q(t),\tau)$ denote the value of $\partial_j F(q,\tau)$ at $t$.  For the remainder of this paper we assume the following regarding the existence of a periodic solution $p$ to \eqref{eq:dae1} and the convergence rate  of its Fourier series.
\begin{Assumption}\label{asp:Fasp}
There exists $T > 0$ and $p \in L^2 \cap C^k$ where:
     \begin{enumerate}
        \item $F(p(t),T)= 0 \quad \forall t \in \mb{R}$.
           \item There exists a sequence $\{K_p(N)\}_{N=0}^{\infty} \subseteq [0,\infty)$ with  $\underset{N\rightarrow \infty}{lim}K_p(N) = 0$ and \begin{equation}\label{eq:pfourierbound}
          \|p-p_N\|_{H^k} \leq K_p(N), \quad p_N := P_N(p), \quad N \in \mb{N}_0.
        \end{equation}
        \item For each $N \in \mb{N}_0$, $\alpha_N$ is the unique element of $X^{2N+1}$ so that $p_N = Q_N(\alpha_N)$.
        \end{enumerate}
\end{Assumption}
We also make the following assumptions on $G$ and $F$, where KP abbreviates ``known period'' and UP abbreviates ``unknown period''.
\begin{Assumption}\label{asp:KPUP}
The DAE \eqref{eq:dae1} satisfies either the KP-assumption or the UP-assumption which are defined as follows:
\begin{enumerate}
    \item  KP-assumption: $T$ is known, $G(\cdot,\hdots,\cdot,t) \in C^1(X^{k+1},X) \text{ } \forall t \in \mb{R}$, and:
    \begin{enumerate}
        \item For all $q \in \tilde{C}^k$ and $j=1,\hdots,k+1$:
        $$F(q,T) \in L^2, \quad \partial_j F(q,T) \in \mc{L}(\tilde{C}^0).$$
        \item For each $j\in \{1,\hdots,k+1\}$ there exist $K_{j,1},\hdots,K_{j,k+1} \geq 0$ so that if $q_1,q_2 \in \tilde{C}^k$ are such that $\|q_n - p\|_{H^k} \leq 1$ for $n=1,2$, then:
    \begin{equation}\label{eq:DFlip1}
    \|[\partial_{j} F(q_1,T) -  \partial_{j}
      F(q_2,T)]\|_{L^2}  \leq T^{j-1}\sum_{l=1}^{k+1} K_{j,l}\|q_1^{(l-1)}-q_2^{(l-1)}\|_{L^2}.
    \end{equation}
    \end{enumerate}
    
    \item UP-assumption: $T$ is an unknown, $G(\cdot,\hdots,\cdot,t) \in C^1(X^{k+1},X)\text{ } \forall t \in \mb{R}$, and:
    \begin{enumerate}
        \item For all $q \in \tilde{C}^k$, $\tau > 0$, and $j=1,\hdots,k+1$:
        $$F(q,\tau) \in L^2, \quad \partial_j F(q,\tau) \in \mc{L}(\tilde{C}^0).$$
        
      \item There exist nonnegative constants $\{K_{j,l,1}\}_{j,l=1}^{k+1}$, $\{K_{j,l,2}\}_{j,l=1}^{k+1}$, $\{K_{j,1}'\}_{j=1}^{k+1}$, $\{K_{j,2}'\}_{j=1}^{k+1}$ so that if $(q_1,\tau_1),(q_2,\tau_2) \in \tilde{C}^k \times (0,\infty)$ are such that $\|q_n - p\|_{H^k} \leq 1$ and $|\tau_n - T| \leq T/2$ for $n=1,2$, then:
    \begin{align}\label{eq:DFlip2a}
      & \|\tau_1^{-(j-1)}\partial_{j} F(q_1, \tau_1)  - \tau_2^{-(j-1)}\partial_{j}F(q_2,\tau_2)\|_{L^2} \\
 \nonumber & \quad \quad \leq \sum_{l=1}^{k+1} K_{j,l,1} \|q_1^{(l-1)} - q_2^{(l-1)}\|_{L^2} + K_{j,1}' |\tau_1-\tau_2|, \\[3pt]
 \label{eq:DFlip2b}
 & \|(j-1) [\tau_1^{-j}\partial_j F(q_1,\tau_1)  q_1^{(j-1)} - \tau_2^{-j}\partial_j F(q_2,\tau_2)  q_2^{(j-1)}]\|_{L^2} \\
 \nonumber & \quad \quad \leq \sum_{l=1}^{k+1} K_{j,l,2}\|q_1^{(l-1)} - q_2^{(l-1)}\|_{L^2} + K_{j,2}'|\tau_1-\tau_2|.
 \end{align}
    \end{enumerate}
\end{enumerate}
\end{Assumption}
We remark that the KP-assumption implies that $G$ is $T$-periodic with respect to $t$ (due to our definition of $L^2$ in Section \ref{sec:notation}) and the UP-assumption implies that $G$ is constant in $t$ (since it implies $F(q,\tau) \in L^2$ for all $\tau > 0$, $q \in \tilde{C}^k$) and is therefore autonomous. The structure of the DAE determines the values of the Lipschitz constants in the KP- and UP-assumptions.  

\begin{remark}\label{rem:embeddingrem}
We remark now on a subtle point regarding the set $V$ defined as all elements $q \in \tilde{C}^k$ such that $\|q-p\|_{H^k}\leq 1$ is bounded in $H^k$.  Since $V$ is bounded in $H^k$, the vector-valued Rellich-Kondrachov Theorem \cite[Theorem 5.1]{Amann2000} implies that it is precompact in $L^2$.  Therefore, letting $Y$ denote an arbitrary Banach space, any function $g_0 \in C^0(V,Y)$ is bounded and any function $g_1 \in C^1(V,Y)$ satisfies a local Lipschitz estimate.
\end{remark}
To unify the analysis for the cases where the period $T$ is known or unknown, we define a function $\mc{F}:\tilde{C}^k \times (0,\infty) \rightarrow L^2$ by:
\begin{equation}\label{eq:defmcF}
\mc{F}(q,\tau) = \left\{\begin{array}{lcr}
F(q,T), \quad \text{$F$ satisfies the KP-assumption} \\[3pt]
F(q,\tau), \quad \text{$F$ satisfies the UP-assumption}
\end{array}\right.
\end{equation}
and we let $\mc{F}(q(t),\tau)$ denote $\mc{F}(q,\tau)$ evaluated at $t \in \mb{R}$.  Assumption \ref{asp:KPUP} implies that $\mc{F}$ is a well-defined map $\tilde{C}^k \times (0,\infty) \rightarrow L^2$ and 
\begin{equation}\label{eq:mcFisC1}
\mc{F} \in C^1(\tilde{C}^k \times (0,\infty), L^2).
\end{equation}
We define some additional notation where $q \in\tilde{C}^k$, $\tau > 0$, and $j \in \{1,\hdots,k+1\}$:
\begin{equation}\label{eq:defofAj}
A_j(q,\tau) = \left\{\begin{array}{lcr}
\partial_j F(q,\tau), \quad \text{$F$ satisfies the UP-assumption} \\[3pt]
\partial_j F(q,T), \quad \text{$F$ satisfies the KP-assumption}
\end{array}\right.
\end{equation}
\begin{equation}\label{eq:defmcF2}
A_{k+2}(q,\tau) = \frac{\partial}{\partial s}\Bigr |_{s=\tau} \mc{F}(q,s).
\end{equation}
We let $A_j(q(t),\tau)$ and $A_{k+2}(q(t),\tau)$ denote $A_j(q,\tau)$ and $A_{k+2}(q,\tau)$ evaluated at $t \in \mb{R}$, respectively.  Assumption \ref{asp:KPUP} and \eqref{eq:mcFisC1} imply that $A_j$ and $A_{k+2}$ are continuous everywhere in their arguments with $A_j(q,\tau) \in \mc{L}(\tilde{C}^0)$ and $A_{k+2}(q,\tau) \in \mc{L}(\mb{R},\tilde{C}^0)$. From this notation and Assumption \ref{asp:KPUP}, it follows that $A_j(q,\tau) x, A_{k+2}(q,\tau) s \in H^1$ for all $q \in \tilde{C}^k$, $\tau \in (0,\infty)$, $x \in X$, and $s \in \mb{R}$.  Fundamental results on the convergence rates of Fourier series imply that there exists a constant $D > 0$ so that:
\begin{equation}\label{eq:ImPbound.1}
\underset{t \in\mb{R}}{sup}\|[I-P_N]A_j(q(t),\tau)\|_{\mc{L}(X)} \leq D N^{-1/2} \|A_j(q,\tau)\|_{H^1}, \quad j=1,\hdots,k+1.
\end{equation}
\begin{equation}\label{eq:ImPbound.2}
\underset{t \in \mb{R}}{sup}\|[I-P_N]A_{k+2}(q(t),\tau)\|_{\mc{L}(\mb{R},X)} \leq D N^{-1/2} \|A_{k+2} (q,\tau)\|_{H^1}.
\end{equation}
We use Estimates \eqref{eq:ImPbound.1}-\eqref{eq:ImPbound.2} to define isolatedness assumptions for $p$.


\subsection{The harmonic balance residual and related functions}

We now describe the HB method.  If $\tau \in (0,\infty)$ and $q \in \tilde{C}^k$, then \eqref{eq:mcFisC1} implies that $\mc{F}(q,\tau) \in L^2$ and therefore:
\begin{equation}
\left\{
\begin{array}{lcr}
\mc{F}(q(t),\tau) = \alpha_0 + \sum_{j=1}^{\infty} \left(\sqrt{2}\alpha_{2j} \cos(2\pi j t) + \sqrt{2}\alpha_{2(j-1)+1} \sin(2\pi j t)\right),\\[3pt]
\alpha_0 = \int_0^1 \mc{F}(q(t),\tau) dt, \quad \alpha_{2j} = \int_0^1 \mc{F}(q(t),\tau) \cos(2\pi jt) dt, \quad j  \in \mb{N}, \\[3pt]
\alpha_{2(j-1)+1} =  \int_0^1 \mc{F}(q(t),\tau) \sin(2\pi jt) dt, \quad j \in \mb{N}.
\end{array}\right.
\end{equation}
Note that the integrals defining the Fourier coefficients are $X$ valued Bochner integrals.  For each $N \in \mb{N}_0$, $x \in X^{2N+1}$, $\tau \in(0,\infty)$, and $j\in \{1,\hdots,k+1\}$ we define several maps as follows:
\begin{equation}\label{eq:defofFNFNsENAN}
\begin{array}{lcr}
\mc{F}_N: X^{2N+1} \times (0,\infty) \rightarrow L^2, \quad  \mc{F}_N(x,\tau) = \mc{F}(Q_N(x),\tau), \\[3pt]
A_{N,j}:X^{2N+1} \times (0,\infty) \rightarrow \mc{L}(H^1), \quad A_{N,j}(x,\tau) = A_j(Q_N(x),\tau), \\[3pt] 
A_{N,k+2}:X^{2N+1} \times (0,\infty) \rightarrow \mc{L}(\mb{R},H^1), \quad A_{N,k+2}(x,\tau) = A_{k+2}(Q_N(x),\tau).
\end{array}
\end{equation}
The above maps are all well-defined due to Assumption \ref{asp:KPUP} and the definitions of $\mc{F}$, $A_j$, and $A_{k+2}$ in Section \ref{sec:dae}.  We let $\mc{F}_N(x,\tau)(t)$, $A_{N,j}(x,\tau)(t)$, and $A_{N,k+2}(x,\tau)(t)$ denote the values of $\mc{F}_N(x,\tau)$, $A_{N,j}(x,\tau)$, and $A_{N,k+2}(x,\tau)$ at $t \in \mb{R}$, respectively.  For $M \in \mb{N}$ and $N \in \mb{N}_0$ we define $c_{M,N},s_{M,N}: X^{2N+1} \times (0,\infty) \rightarrow X$ by: 
\begin{equation}\label{eq:csdef}
\begin{array}{lcr}
c_{0,N}(x,\tau)   = \int_0^1 \mc{F}_N(x,\tau)(t) dt, \quad c_{M,N}(x,\tau) = \frac{2\sqrt{2}}{2}\int_0^1 \mc{F}_N(x,\tau)(t) \cos(2\pi M t)dt,\\[3pt]
s_{M,N}(x,\tau) = \frac{2\sqrt{2}}{2}\int_0^1 \mc{F}_N(x, \tau)(t) \sin(2\pi M t)dt.
\end{array}
\end{equation}
We define the HB residual $R_{N}: X^{2N+1} \times (0,\infty)\rightarrow X^{2N+1}$ by $R_0(x,\tau) = c_{0,0}(x,\tau)$ when $N = 0$, otherwise:
\begin{equation}\label{eq:Rn}
R_{N}(x,\tau) = (c_{0,N}(x,\tau),s_{1,N}(x,\tau),c_{1,N}(x,\tau),\hdots,s_{N,N}(x,\tau),c_{N,N}(x,\tau)), \quad N > 0.
\end{equation}
If $T$ is known, then the harmonic balance method approximates $T$-periodic solutions to \eqref{eq:dae1} by approximating a zero of $R_{N}(\cdot,T)$ for sufficiently large $N$.  If $T$ is unknown, then the harmonic balance method approximates a zero of $\hat{R}_{N}(\cdot,\cdot) = (R_{N}(\cdot,\cdot),\sigma(\cdot))$ for sufficiently large $N$, where $\sigma$ is a phase condition (see Section \ref{sec:unknown}).

In Lemma \ref{lem:RisC1} (Section \ref{sec:sharedestimates}) we show that $\mc{F}_N = \mc{F}_N(x,\tau)$, $c_{M,N}  = c_{M,N}(x,\tau)$, $s_{M+1,N} = s_{M+1,N}(x,\tau)$, and $R_N = R_N(x,\tau)$ are continuously differentiable everywhere for all $M,N \in \mb{N}_0$.  We denote the partial derivatives of $R_N = R_N(x,\tau)$ at $(x_0,\tau_0)$ with respect to $x$ and $\tau$ by $\partial_1 R_N(x_0,\tau_0)$ and $\partial_2 R_N(x_0,\tau_0)$, respectively.  We denote the total or Fr\'echet derivative of $R_N$ at $(x_0,\tau_0)$ by $D R_N(x_0,\tau_0)$.  Similar notation is used to denote the partial and Fr\'echet derivatives of $\mc{F}_N$, $c_{M,N}$ and $s_{M+1,N}$ with respect to $x$, $\tau$, and $(x,\tau)$.

\section{Convergence estimates}\label{sec:convergence}


\subsection{Shared estimates and lemmas}\label{sec:sharedestimates}

In this section (Section \ref{sec:sharedestimates}) we define several constants, outline the proofs of Theorems \ref{thm:known}-\ref{thm:unknown}, prove several fundamental estimates and lemmas, and state an inexact Newton-Kantorovich theorem (Theorem \ref{thm:inexactNK}) and a variant of the Lax-Milgram theorem (Theorem \ref{thm:laxmilgram}). These results are used to prove Theorems \ref{thm:known}-\ref{thm:unknown} in Sections \ref{sec:known}-\ref{sec:unknown}.

\subsubsection{Constants defined by $\mc{F}$ and outline of the proofs of main theorems}\label{sec:sharedconstants}

We define several constants related to $\mc{F}$ and subsequently outline the proofs of Theorems \ref{thm:known}-\ref{thm:unknown}.  Assumptions \ref{asp:Fasp}-\ref{asp:KPUP} and the definitions and properties of $\mc{F}$, $A_j$, and $A_{k+2}$ (see Section \ref{sec:dae}) imply that the following constants are well-defined for all $\ve > 0$ and $N\in \mb{N}_0$:
\begin{itemize}
     \item Define $J_1, J_2 > 0$ so that if $(q,\tau) \in \tilde{C}^k \times (0,\infty)$ is such that $\|q-p\|_{H^k} \leq 1$ and $|\tau - T| \leq T/2$, then:
      \begin{equation}\label{eq:defJ1J2}
        \|\mc{F}(q,\tau)\|_{L^2} \leq J_1 \|q-p\|_{H^k} + J_2 |\tau-T|.
      \end{equation}

    \item Define $\delta(\ve)$ to be the largest number in $(0,\text{min}\{1,T/2,\ve\}]$ so that if $(q,\tau) \in \tilde{C}^k \times (0,\infty)$ is such that $\|q-p\|_{H^k} \leq \delta(\ve)$ and $|\tau - T| \leq \delta(\ve)$, then:
    \begin{equation}\label{eq:AFest}
      \begin{array}{c}
        \underset{t \in \mb{R}}{\sup}\|A_j(q(t),\tau)-A_j(p(t),T)\|_{\mc{L}(X)} \leq \ve/2, \quad j=1,\hdots,k+1\\[3pt]
        \underset{t \in \mb{R}}{\sup}\|A_{k+2}(q(t),\tau) - A_{k+2}(p(t),T)\|_{\mc{L}(\mb{R},X)} \leq \ve/2.
      \end{array}
    \end{equation}

    \item Define $\zeta_N(\ve)$ and $\delta_N(\ve)$ by:
      \begin{equation}\label{eq:defdeltaepsN}
        \zeta_N(\ve) =\frac{\ve}{2}\left(\sum_{j=0}^{k}(2\pi N)^j \right)^{-1} , \quad \delta_N(\ve) = \zeta_N(\delta(\ve)).
      \end{equation}


\end{itemize}
This next set of constants is well-defined for all $\ve > 0$ due to Assumptions \ref{asp:Fasp}-\ref{asp:KPUP}, the definitions and properties of $\mc{F}$, $A_j$, and $A_{k+2}$ in Section \ref{sec:dae}, Estimates \eqref{eq:ImPbound.1}-\eqref{eq:ImPbound.2}, and the definitions of the constants $\delta(\ve)$ and $K_{A,1},\hdots,K_{A,k+1}$ defined above:
\begin{itemize}
    \item Define $N_1(\ve)$ to be the least element of $\mb{N}_0$ so that if $N \geq N_1(\ve)$ then:
      \begin{equation}\label{eq:N1def.1}
         K_p(N) \leq \ve/2
         \end{equation}
and if $(q,\tau) \in \tilde{C}^k \times (0,\infty)$ are such that $\|q-p\|_{H^k} \leq \delta(\ve)$ and $|\tau-T| \leq \delta(\ve)$, then for $j\in\{1,\hdots,k+1\}$:
\begin{equation}\label{eq:N1def.2}
\begin{array}{lcr}
         \underset{t\in\mb{R}}{\sup} \|[I-P_N] A_j(q(t),\tau)\|_{\mc{L}(X)}\leq \ve/2, \quad \\[3pt]
           \underset{t\in\mb{R}}{\sup} \|[I-P_N]A_{k+2}(q(t),\tau)\|_{\mc{L}(\mb{R},X)} \leq \ve / 2.
         \end{array}
      \end{equation}

    \item  If \eqref{eq:dae1} satisfies the KP-assumption, then define $M_1$ to be the maximum of the following set:
    \begin{equation}\label{eq:defm1}
       \{0\} \cup\{j+l-2 : j,l \in \{1,\hdots,k+1\} \text{ and } K_{j,l} > 0\}.
    \end{equation} 
    If \eqref{eq:dae1} satisfies the UP-assumption, then define $M_2$ to be the maximum of the following set:
    \begin{align}\label{eq:defm2}
      \{0\} & \cup \{j+l-2: j,l \in \{1,\hdots,k+1\} \text{ and } K_{j,l,1} > 0\}  \\
      \nonumber & \cup \{j-1 : j \in \{1,\hdots,k+1\} \text{ and } K_{j,1}' > 0 \} \\
      \nonumber & \cup \{l-1 : l \in \{1,\hdots,k+1\} \text{ and }  K_{j,l,2} > 0 \text{ for some } j\}.
    \end{align}

    \item     If \eqref{eq:dae1} satisfies the KP-assumption, then define $\rho_1 \in \mb{N}_0$ to be the smallest number such that the following holds:
    \begin{equation}\label{eq:rho1def}
      1 + (2\pi N) + \hdots (2\pi N)^k \leq (2\pi)^k(N+1)^{M_1 + \rho_1}, \quad N \in \mb{N}_0.
    \end{equation}
    The constant $\rho_1$ is used to define a sufficient condition on the convergence rate $K_p(N)$ in the hypotheses of Theorem \ref{thm:known}.  Another constant, $\rho_2$ is defined in Section \ref{sec:unknown} for use in the hypotheses of Theorem \ref{thm:unknown}.

\end{itemize}

Theorems \ref{thm:known} and \ref{thm:unknown} are proved by applying an inexact Newton-Kantorovich result (Theorem \ref{thm:inexactNK}) to $R_N(\cdot,T)$ and $\hat R_N = (R_N,\sigma_N) = (R_N(x,\tau),\sigma_N(x))$, respectively, where $\sigma_N$ is an appropriately defined phase condition (see Section \ref{sec:unknown}).  We briefly outline how the hypotheses of Theorem \ref{thm:inexactNK} are determined by the constants and assumptions defined above and in Sections \ref{sec:preliminaries} and \ref{sec:known}-\ref{sec:unknown}.  Fix some $\ve > 0$:

\begin{itemize}
   \item Given an initial guess $x^0$ (resp. $(x^0,\tau^0)$) for a zero of $R_N(\cdot,T)$ (resp. $\hat R_N$), the values of $J_1,J_2$ determine bounds in the proof of Theorem \ref{thm:known} (resp. Theorem \ref{thm:unknown}) for $\|R_N(x^0,T)\|$ (resp. $\|\hat R_N(x^0,\tau^0)\|$) in terms of  $\|Q_N(x^0)-p_N\|_{H^k}$ and $K_p(N)$ (resp. $\|Q_N(x^0)-p_N\|_{H^k}$, $K_p(N)$, and $|\tau^0-T|$).  

   \item The constants $\zeta_N(\ve)$ and $\delta_N(\ve)$ determine bounds on $\|Q_N(x^0)-p_N\|_{H^k}$ and the constants $N_1(\ve)$ and $N_1(\delta(\ve))$ determine bounds on $K_p(N)$.  We prove explicit bounds using Lemmas \ref{lem:Qlem1}-\ref{lem:Qlem2}.

  \item In Lemmas \ref{lem:RisC1}-\ref{lem:DxRlem} we show that the operators $\partial_1 R_N$ and $D R_N$ exist and are bounded.  Based on this, in Section \ref{sec:unknown} we show that the operator $D \hat R_N$ exists and is bounded.  Using Lemma \ref{lem:Alem1} together with Assumption \ref{asp:known} (resp. Assumption \ref{asp:unknown}) and Lemma \ref{lem:DxRinvboundknown} (resp. Lemma \ref{lem:DxRhatlem}), we show that $\partial_1 R_N(\cdot,T)$ (resp. $D\hat R_N$) is invertible with a bounded inverse.
  
  \item In Lemma \ref{lem:DxxRlem} we show that $\partial_1 R_N(\cdot,T)$ (resp. $D R_N$) is Lipschitz with Lipschitz constant $L_1 N^{M_1}$ (resp. $L_2 N^{M_2}$).  The constants $M_1$ and $M_2$ are defined above (\eqref{eq:defm1}-\eqref{eq:defm2}), $L_1$ (see \eqref{eq:DxxRlem.2a}) is defined from the set of constants $\{K_{j,l}\}_{j,l=1}^{k+1}$ from the KP-assumption, and $L_2$ (see \eqref{eq:defL2}) is defined from the sets of constants $\{K_{j,l,n}\}_{j,l=1,n=1}^{k+1,2}$ and $\{K_{j,n}'\}_{j=1,n=1}^{k+1,2}$ from the UP-assumption. 
  
\end{itemize}
The proofs of Theorem \ref{thm:known}-\ref{thm:unknown} essentially combine the above points (which we prove in Sections \ref{sec:Qlem}-\ref{sec:RFAlem} and \ref{sec:known}-\ref{sec:unknown}) so that the hypotheses of Theorem \ref{thm:inexactNK} are satisfied by $R_N(\cdot,T)$ and $\hat R_N$ for all sufficiently large $N$ and sufficiently accurate initial guesses. 

\subsubsection{Lemmas for $Q_N$}\label{sec:Qlem}

We now prove several lemmas related to $Q_N$ (defined by \eqref{eq:QN}).  In the following lemma we show that $DQ_N^{(j)}(x)$ exists for all $x \in X^{2N+1}$ and $j \in \mb{N}_0$ and establish several equalities and bounds related to $Q_{N}^{(j)}$ and $D Q_{N}^{(j)}$.
\begin{Lemma}\label{lem:Qlem1}
Let $j,N \in \mb{N}_0$.  Then $Q_{N}^{(j)} \in C^1(X^{2N+1},C^{\infty})$ and if $x,y \in X^{2N+1}$, then the following hold:
\begin{equation}\label{eq:Qlem1.0}
\|Q_{N}(x) - Q_{N}(y)\|_{L^2} = \|x-y\|, \quad \|Q_{N}(x)\|_{L^2} = \|x\|,
\end{equation}
\begin{equation}\label{eq:Qlem1.1}
\|Q_{N}^{(j)}(x) - Q_{N}^{(j)}(y)\|_{L^2} \leq (2\pi N)^j  \|x-y\|, \quad \|Q_N^{(j)}(x)\|_{L^2} \leq (2\pi N)^j \|x\|,
\end{equation}
\begin{equation}\label{eq:Qlem1.2}
DQ_N^{(j)}(y)x = Q_N^{(j)}(x), \quad \| DQ_{N}^{(j)}(x)\|_{L^2} \leq (2\pi N)^j.
\end{equation}
\end{Lemma}
\begin{proof}
The equalities in \eqref{eq:Qlem1.0} and inequalities in \eqref{eq:Qlem1.1} are a straightforward consequence of the definition of $Q_N$ and the choice of the norm $\|\cdot\|$ on $X^{2N+1}$.  Note that since $Q_{N}^{(j)} \in \mc{L}(X^{2N+1},C^{\infty})$ it follows that $Q_N^{(j)} \in C^1(X^{2N+1},C^{\infty})$ with $DQ_N^{(j)}(y)x = Q_N^{(j)}(x)$, establishing the equality in \eqref{eq:Qlem1.2}.  If $0 \neq w \in X^{2N+1}$, then  $DQ_N^{(j)}(x)w = Q_N^{(j)}(w)$ and \eqref{eq:Qlem1.1} then implies that:
$$\|DQ_{N}^{(j)}(x)w\|_{L^2} = \|Q_{N}^{(j)}(w)\|_{L^2} \leq (2\pi N)^j\|w\|.$$
Therefore $\|DQ_{N}^{(j)}\|_{L^2} \leq (2\pi N)^j$ which implies the estimates in \eqref{eq:Qlem1.1}.
\end{proof}
Lemma \ref{lem:Qlem1} together with Estimate \eqref{eq:pfourierbound} and the triangle inequality imply the following estimate for any $N \in \mb{N}_0$ and $x \in X^{2N+1}$:
\begin{equation}\label{eq:Qpbound}
\| Q_{N}(x) - p\|_{H^k} \leq K_p(N) + \sum_{j=0}^{k}(2\pi N)^j\|x - \alpha_N\|.
\end{equation}
The next lemma estimates how well $Q_N(x)$ approximates $p$ when $x$ is close to $\alpha_N$ and $N$ is sufficiently large.
\begin{Lemma}\label{lem:Qlem2}
If $\ve > 0$, $N \geq N_1(\ve)$, and $x \in \mc{B}_N(\alpha_N,\zeta_N(\ve))$, then:
\begin{equation}\label{eq:Qlem2.1}
\|Q_N(x) - p\|_{H^k} \leq \ve.
\end{equation}
\end{Lemma}
\begin{proof}
We have:
\begin{align*}
\|Q_N(x)-p\|_{H^k} & \leq \sum_{j=0}^{k}(2\pi N)^j \|x - \alpha_N\| + K_p(N) && \text{(Estimate \eqref{eq:Qpbound})}\\
& = \ve \zeta_N(\ve)^{-1}\|x-\alpha_N\|/2 + K_p(N) && \text{(Definition of $\xi_N(\ve)$ in \eqref{eq:defdeltaepsN})}\\
& \leq \ve/2 + K_p(N) && \text{($\|x-\alpha_N\| \leq \zeta_N(\ve)$)} \\
& \leq \ve/2 + \ve/2 = \ve. && \text{(Estimate \eqref{eq:N1def.1}, $N \geq N_1(\ve)$)}
\end{align*}
\end{proof}

\subsubsection{Lemmas for $R_N$, $\mc{F}_N$, $A_{N,j}$, and $A_{N,k+2}$}\label{sec:RFAlem}

We next establish some estimates and lemmas related to $R_N$, $\mc{F}_N$, $A_{N,j}$, and $A_{N,k+2}$.  The following lemma determines bounds on the differences $P_N A_{N,j}(x,\tau)-A_j(p,T)$ and $P_N A_{N,k+2}(x,\tau)-A_{k+2}(p,T)$ in a form useful for establishing isolatedness conditions in Sections \ref{sec:known}-\ref{sec:unknown}.
\begin{Lemma}\label{lem:Alem1}
If $\ve >0$, $N \geq N_1(\delta(\ve))$, $x \in \mc{B}_N(\alpha_N,\delta_N(\ve))$, and $|\tau - T| \leq \delta(\ve)$, then:
\begin{equation}\label{eq:Alem1.2}
\begin{array}{lcr}
\underset{t \in \mb{R}}{\sup}\|P_N A_{N,j}(x,\tau)(t) - A_j(p(t),T)\|_{\mc{L}(X)} \leq \ve, \quad j=1,\hdots,k+1 \\[3pt]
\underset{t \in \mb{R}}{\sup}\|P_N A_{N,k+2}(x,\tau)(t)-A_{k+2}(p(t),T)\|_{\mc{L}(\mb{R},X)} \leq \ve.
\end{array}
\end{equation}
\end{Lemma}
\begin{proof}
Since $\delta_N(\ve)= \zeta_N(\delta(\ve))$, Lemma \ref{lem:Qlem2} implies that:
\begin{equation}\label{eq:Alem1.3}
\|Q_N(x)-p\|_{H^k} \leq \delta(\ve).
\end{equation}
Estimate \eqref{eq:AFest}, using Estimate \eqref{eq:Alem1.3} and $|\tau-T| \leq \delta(\ve) \leq T/2$, then implies that:
\begin{equation}\label{eq:Alem1.4}
\underset{t \in [0,1]}{\sup}\|A_{N,j}(x,\tau)(t)-A_j(p(t),T)\|_{\mc{L}(X)} \leq \ve/2, \quad j =1,\hdots,k+1.
\end{equation}
We then have the following estimates:
\begin{align*}
& \underset{t \in \mb{R}}{\sup} \|P_{N}  A_{N,j}(x,\tau)(t) - A_j(p(t),T)\|_{\mc{L}(X)} \\[3pt]
& \quad \quad \leq \underset{t \in \mb{R}}{\sup}\|[I-P_N]A_{N,j}(x,\tau)(t)\|_{\mc{L}(X)} \\[3pt]
& \quad \quad \quad \quad + \underset{t \in \mb{R}}{\sup}\|A_{N,j}(x,\tau)(t)-A_j(p(t),T)\|_{\mc{L}(X)} && \text{(Triangle Inequality)}\\[3pt]
& \quad \quad \leq \underset{t \in [0,1]}{\sup}\|[I-P_N]A_{N,j}(x,\tau)(t)\|_{\mc{L}(X)} + \ve/2 && \text{(Estimate  \eqref{eq:Alem1.4})} \\
& \quad \quad\leq \ve/2 + \ve/2 = \ve  && \text{($N \geq N_1(\delta(\ve))$, Estimate \eqref{eq:N1def.2})}. 
\end{align*}
Proof of the second estimate in \eqref{eq:Alem1.2} is similar and is omitted for brevity.
\end{proof}
The next lemma establishes some facts related to the existence and relationships of the derivatives of $\mc{F}_N$, $R_{N}$, $c_{M,N}$, and $s_{M,N+1}$ for $M,N \in \mb{N}_0$.
\begin{Lemma}\label{lem:RisC1}
Let $M,N \in \mb{N}_0$, $x,y \in X^{2N+1}$, $\tau \in (0,\infty)$, and $s \in \mb{R}$. We have:
\begin{align}\label{eq:RisC1.a}
& \mc{F}_{N} \in C^1(X^{2N+1} \times (0,\infty),C^1), \quad c_{M,N},s_{M+1,N} \in C^1(X^{2N+1} \times (0,\infty),X), \\[3pt]
& \label{eq:RisC1.b} R_{N} \in C^1(X^{2N+1} \times (0,\infty),X^{2N+1}),\\
\label{eq:RisC1.c}
& Q_N(R_N(x,\tau)) = P_N \mc{F}_N(x,\tau), \\
\label{eq:RisC1.d}
& Q_N([\partial_1 R_N(x,\tau)] y) = P_N([\partial_1 \mc{F}_N(x,\tau)]y),\\
\label{eq:RisC1.e} 
& Q_N([\partial_2 R_N(x,\tau)] s) =  P_N([\partial_2 \mc{F}_N(x,\tau)] s) = P_N(A_{N,k+2}(x,\tau)s).
\end{align}
If $M \geq 1$, then the following equalities always hold:
\begin{equation}\label{eq:RisC1.1}
\begin{array}{lcr}
[\partial_1 \mc{F}_N(x,\tau)] y = \sum_{j=1}^{k+1}\tau^{-(j-1)} A_{N,j}(x,\tau) Q_N^{(j-1)}(y), \\[3pt]
[\partial_1 c_{0,N}(x,\tau)] y = \int_0^1 \sum_{j = 1}^{k+1} \tau^{-(j-1)} A_{N,j}(x,\tau)(t) Q_N^{(j-1)}(y)(t) dt, \\[3pt]
[\partial_1 c_{M,N}(x,\tau)] y = \frac{2\sqrt{2}}{2}\int_0^1 \sum_{j = 1}^{k+1} \tau^{-(j-1)} A_{N,j}(x,\tau)(t)  Q_N^{(j-1)}(y)(t) \cos(2\pi M t)dt, \\[3pt]
[\partial_1 s_{M,N}(x,\tau)] y = \frac{2\sqrt{2}}{2} \int_0^1 \sum_{j = 1}^{k+1} \tau^{-(j-1)} A_{N,j}(x,\tau)(t)  Q_N^{(j-1)}(y)(t) \sin(2\pi M t)dt,\\[3pt]
\end{array}
\end{equation}
and in addition:
\begin{equation}\label{eq:RisC1.2}
\begin{array}{lcr}
[\partial_2 \mc{F}_N(x,\tau)] s =  A_{N,k+2}(x,\tau)s, \\[3pt]
[\partial_2 c_{0,N}(x,\tau)] s = \int_0^1 A_{N,k+2}(x,\tau) s dt, \\[3pt]
[\partial_2 c_{M,N}(x,\tau)] s = \frac{2\sqrt{2}}{2}\int_0^1 A_{N,k+2}(x,\tau)(t) \cos(2\pi M t) s dt, \\[3pt]
[\partial_2 s_{M,N}(x,\tau)] s = \frac{2\sqrt{2}}{2}\int_0^1 A_{N,k+2}(x,\tau)(t) \sin(2\pi M t) s dt .
\end{array}
\end{equation}
\end{Lemma}
\begin{proof}
The definitions of $c_{M,N}$, $s_{M+1,N}$, and $R_N$ together with \eqref{eq:mcFisC1} imply \eqref{eq:RisC1.a}-\eqref{eq:RisC1.b}.    The first equality of \eqref{eq:RisC1.1} follows from the definitions of $\mc{F}$ and $A_{N,1},\hdots,A_{N,k+1}$.  If $M \geq 1$, then we obtain the latter three equalities of \eqref{eq:RisC1.1} by differentiating the Fourier integrals defining $c_{0,N}$, $c_{M,N}$, and $s_{M,N}$ under the integral (this is well-defined due to \eqref{eq:RisC1.a}).  The equalities of \eqref{eq:RisC1.2} are proved similarly.  Equalities \eqref{eq:RisC1.c}-\eqref{eq:RisC1.e} are consequences of \eqref{eq:RisC1.1}-\eqref{eq:RisC1.2} and definitions of $R_N$, $Q_N$, and $P_N$. 
\end{proof}
The next lemma relates the solution of the linear equation $[\partial_1 R_N(x,\tau)] y = z$
to a linear DAE defined in terms of the coefficient matrices $A_{N,1},\hdots,A_{N,k+1}$ and, in addition, relates $A_{N,k+2}$ to $\partial_2 R_N(x,\tau)$ and establishes that $\partial_1 R_N(x,\tau)$ and $\partial_2 R_N(y,\tau)$ are bounded linear operators.
\begin{Lemma}\label{lem:DxRlem}
Let $x,y,z \in X^{2N+1}$, $\tau \in (0,\infty)$.  Then $[\partial_1 R_{N}(x,\tau)]y  = z$ if and only if $Q_{N}(y)$ is a $1$-periodic solution of the following linear DAE:
\begin{equation}\label{eq:DxRlem.1}
\sum_{j=1}^{k} \tau^{-(j-1)} P_{N}  A_{N,j+1}(x,\tau)  u^{(j-1)} = Q_{N}(z).
\end{equation}
Both $\partial_1 R_N(x,\tau)$ and $\partial_2 R_N(x,\tau)$ are bounded linear operators.
\end{Lemma}
\begin{proof}
Lemma \ref{lem:RisC1} implies that:
$$Q_N([\partial_1 R_N(x,\tau)]y) = P_N [\partial_1 \mc{F}_N(x,\tau)]y = \sum_{j=1}^{k} \tau^{-(j-1)} P_{N}  A_{N,j+1}(x,\tau)  Q_N(y)^{(j-1)}.$$
Since $Q_N$ is 1-to-1, the above equality implies that $[\partial_1 R_N(x,\tau)]y = z$ if and only if $Q_N(y)$ is a $1$-periodic solution to \eqref{eq:DxRlem.1}.  To prove boundedness of $\partial_1 R_N(x,\tau)$, we deduce that:
\begin{align*}
\|[\partial_1 R_N(x,\tau)]y\| & = \|Q_N([\partial_1 R_N(x,\tau)]y)\|_{L^2} && \text{(Lemma \ref{lem:Qlem1})} \\
& = \|\sum_{j=1}^{k+1} P_N A_{N,j}(x,\tau) Q_N^{(j-1)}(y)\|_{L^2} && \text{(Lemma \ref{lem:RisC1})}\\
 & \leq \sum_{j=1}^{k+1} (2\pi N)^{j-1} \|A_{N,j}(x,\tau)\|_{L^2}\|y\|. && \text{(Estimate \eqref{eq:besselbound} and Lemma \ref{lem:Qlem1})}
\end{align*}
Assumption \ref{asp:KPUP} implies that $A_{N,j}(x,\tau) \in \mc{L}(L^2)$ so that $\|A_{N,j}(x,\tau)\|_{L^2} < \infty$ for $j=1,\hdots,k+1$ and therefore $\partial_1 R_N(x,\tau)$ is a bounded linear operator.  The proof that $\partial_2 R_N(x,\tau)$ is a bounded linear operator follows similarly.
\end{proof}
The next lemma bounds $R_N$ in terms of $\mc{F}_N$ and establishes Lipshitz bounds on $\partial_1 R_N(\cdot, T)$ and $DR_N$ when \eqref{eq:dae1} satisfies the KP-assumption and UP-assumption, respectively.
\begin{Lemma}\label{lem:DxxRlem}
If $N \in \mb{N}_0$ and $(x,\tau) \in X^{2N+1} \times (0,\infty)$, then:
\begin{equation}\label{eq:DxxRlem.1}
\|R_{N}(x,s)\| \leq \|\mc{F}_N(x,s)\|_{L^2}.
\end{equation}
Let $\ve \in (0,1]$ and $y,z \in \mc{B}(\alpha_N,\zeta_N(\ve))$ and suppose that $N \geq N_1(\delta(\ve))$.  If \eqref{eq:dae1} satisfies the KP-assumption, then:
\begin{equation}\label{eq:DxxRlem.2a}
\|\partial_1 R_{N}(y,T) - \partial_1 R_{N}(z,T)\|\leq  L_1 N^{M_1}\|y-z\|, \quad
L_1 := \sum_{j,l=1}^{k+1} K_{j,l} (2\pi)^{j+l-2}.
\end{equation}
If \eqref{eq:dae1} satisfies the UP-assumption and $s,w \in \{a : |a-T| \leq T/2\}$, then:
\begin{equation}\label{eq:DxxRlem.2b}
\|D R_{N}(y,s) - D  R_{N}(z,w)\| \leq  L_2 N^{M_2}\|(y,s)-(z,w)\|
\end{equation}
where:
\begin{equation}\label{eq:defL2}
L_2 =  \sum_{j=1}^{k+1} \left(\sum_{l=1}^{k+1}K_{j,l,1}(2\pi)^{j+l-2} + K_{j,l,2}  (2\pi)^{l-1}\right) + K_{j,1}' (2\pi)^{j-1} + K_{j,2}'.
\end{equation}
\end{Lemma}
\begin{proof}
If $N \in \mb{N}_0$ and $(x,\tau) \in X^{2N+1} \times (0,\infty)$, then:
\begin{align*}
\|R_N(x,\tau)\| & = \|Q_N(R_N(x,\tau))\|_{L^2} && \text{(Lemma \ref{lem:Qlem1})} \\
& = \|P_N(\mc{F}_N(x,\tau))\|_{L^2} && \text{(Lemma \ref{lem:Qlem1})} \\
& \leq \|F_{N}(x,\tau)\|_{L^2} && \text{(Estimate \ref{eq:besselbound})}
\end{align*}
which proves Estimate \eqref{eq:DxxRlem.1}. Next, we assume that $\ve \in (0,1]$, $N \geq N_1(\delta(\ve))$, $y,z \in \mc{B}(\alpha_N,\zeta_N(\ve))$, and $s,w \in \{a: |a-T| \leq T/2\}$.  Equations \eqref{eq:RisC1.d}-\eqref{eq:RisC1.e} of Lemma \ref{lem:RisC1} imply that:
\begin{equation}\label{eq:DxxRlem.3a}
\|\partial_j R_N(y,s)-\partial_j R_N(z,w)\| \leq \|\partial_j \mc{F}_N(y,s)- \partial_j \mc{F}_N(z,w)\|_{L^2}, \quad j =1,2.
\end{equation}
Additionally, Lemma \ref{lem:Qlem2} implies that:
\begin{equation}\label{eq:DxxRlem.3}
\|Q_N(y)-p\|_{H^k},|Q_N(z)-p\|_{H^k} \leq \ve \leq 1. 
\end{equation}
Assume that \eqref{eq:dae1} satisfies the KP-assumption.  If $0 \neq \lambda \in X^{2N+1}$, then:
\begin{align*}
\|[\partial_1 & \mc{F}_N(y,T)  - \partial_1 \mc{F}_N(z,T)]\lambda\|_{L^2}/\|\lambda\| \\
 & \leq \sum_{j=1}^{k+1}T^{-(j-1)} \|[A_{N,j}(y,T) - A_{N,j}(z,T)] Q_N^{(j-1)}(\lambda)\|_{L^2}/\|\lambda\|  && \text{(Equation \eqref{eq:RisC1.1})}\\
 & \leq \sum_{j=1}^{k+1} T^{-(j-1)}\|[A_{N,j}(y,T)   - A_{N,j}(z,T)] \|_{L^2}(2\pi N)^{j-1}  && \text{(Lemma \ref{lem:Qlem1})}\\
 & \leq \sum_{j,l=1}^{k+1} K_{j,l} (2\pi N)^{j-1}\|Q_N^{(l-1)}(y)-Q_N^{(l-1)}(z)\|_{L^2} && \text{(\eqref{eq:DFlip1} using \eqref{eq:DxxRlem.3})} \\
  & \leq \sum_{j,l=1}^{k+1} K_{j,l} (2\pi )^{j+l-2} N^{M_1}\|y-z\|. && \text{(Lemma \ref{lem:Qlem1} \& \eqref{eq:defm1})}
\end{align*}
Therefore $\|\partial_1 \mc{F}_N(y,T)  - \partial_1 \mc{F}_N(z,T)\|_{L^2} \leq L_1 N^{M_1}\|y-z\|$ which, combined with Estimate \eqref{eq:DxxRlem.3a}, implies Estimate \eqref{eq:DxxRlem.2a}.  Now assume that \eqref{eq:dae1} satisfies the UP-assumption.  Note that since the UP-assumption implies that $G$ is independent of $t$ (see the remarks below Assumption \ref{asp:KPUP}), and therefore:
\begin{equation}\label{eq:DxxRlem.special}
A_{N,k+2}(x,\tau) = -\sum_{j=1}^{k+1}(j-1)\tau^{-j} A_{N,j}(x,\tau) Q_N^{(j-1)}(x).
\end{equation}
If $(0,0) \neq (\lambda,b) \in X^{2N+1} \times \mb{R}$, then the triangle inequality and Estimate \eqref{eq:DxxRlem.3} imply that:
\begin{align*}
\|[D & R_N(y,w)  - D R_N(z,w)](\lambda,b)\|_{L^2} \\[3pt]
& \leq \|[\partial_1 \mc{F}_N(y,s)  - \partial_1 \mc{F}_N(z,w)]\lambda\|_{L^2} + \|[\partial_2 \mc{F}_N(y,s)  - \partial_2 \mc{F}_N(z,w)]b\|_{L^2}.
\end{align*}
If \eqref{eq:dae2} satisfies the UP-assumption, then the assumptions on $\ve$, $N$, $y,z$, and $s,w$ and similar reasoning as above, using Estimates \eqref{eq:DFlip2a}-\eqref{eq:DFlip2b} rather than \eqref{eq:DFlip1} and  \eqref{eq:RisC1.2} and \eqref{eq:DxxRlem.special} in addition to \eqref{eq:RisC1.1},  implies the following estimates:
\begin{align}\label{eq:DxxRlem.4.1}
\|[\partial_1 & \mc{F}_N(y,s)  - \partial_1 \mc{F}_N(z,s)]\|_{L^2} \\
\nonumber & \leq \left(\sum_{j,l=1}^{k+1} K_{j,l,1} (2\pi)^{j+l-2} N^{j+l-2}+ K_{j,1}' (2\pi)^{j-1} N^{j-1}\right)\|(y,s)-(z,w)\|.
\end{align}
\begin{align}\label{eq:DxxRlem.4.2}
\|\partial_2 & \mc{F}_N(y,s)  - \partial_2 \mc{F}_N(z,w)\|_{L^2} \\
\nonumber & \leq \left(\sum_{j,l=1}^{k+1} K_{j,l,2}(2\pi)^{l-1} N^{l-1} + K_{j,2}' \right)\|(y,s)-(z,w)\|.
\end{align}
Estimate \eqref{eq:DxxRlem.2b} is established from Estimates \eqref{eq:DxxRlem.4.1}-\eqref{eq:DxxRlem.4.2} using the definition of $M_2$ in \eqref{eq:defm2}.
\end{proof}
We close this section by stating an inexact  Newton-Kantorovich Theorem and a Lax-Milgram Theorem that are needed in Sections \ref{sec:known}-\ref{sec:unknown}.
\begin{Theorem}\label{thm:inexactNK}[Theorem  6.3 of \cite{FerreiraSvaiter2012}]
Let $Y$ be a Banach space with norm $\|\cdot\|_Y$ and for each $y \in Y$ and $\delta \geq 0$ let $B(y,\delta):=\{x \in Y: \|x-y\|_Y \leq \delta\}$.  Suppose that $G \in C^1(U,Y)$ where $U \subseteq Y$ is an open set.  Assume $y_0 \in U$ is such that the Fr\'echet derivative of $G$, denoted by $DG$, is invertible at $y_0$ and there exists $\beta,L > 0$ with $\beta L < 1/2$ and $B(y_0,L^{-1}) \subseteq U$ with:
\begin{enumerate}
    \item $\|DG(y_0)^{-1}(DG(z)-DG(y))\|_Y \leq L\|z-y\|_Y, \quad  y,z \in B(y_0,L^{-1}).$
    \item $\|DG(y_0)^{-1}G(y_0)\| \leq \beta$
\end{enumerate}
Let $\theta \geq 0$ be such that $\theta \leq (1-\sqrt{2\beta L})/(1+\sqrt{2 \beta L})$.  Then the sequence $\{y_j\}_{j=0}^{\infty}$ generated by the following inexact Newton process:
\begin{equation}\label{eq:inexactNK.1}
y_{j+1} = y_j + \delta_j, \quad DG(y_j)\delta_j = -G(y_j) + r_j, \quad j \in \mb{N}_0
\end{equation}
is well-defined  and $ \|DG(y_0)^{-1}G(y_j)\|_Y \leq \frac{1}{2}(1+\theta^2)^j \beta$ for $j \in \mb{N}_0$ whenever the following holds:
\begin{equation}\label{eq:inexactNK.2}\|DG(y_0)^{-1}r_j\|_Y \leq \theta\|DG(y_0)^{-1} G(y_j)\|_Y, \quad j \in \mb{N}_0.
\end{equation}
Furthermore, the sequence $\{y_j\}_{j=0}^{\infty}$ is contained in  $B(y_0,\sqrt{2\beta L}/L)$ and there exists $y_* \in B(y_0,(1-\sqrt{1-2 L \beta})/L)$ where $y_*$ is the unique zero of $F$ in $B(y_0,L^{-1})$ and satisfies the following  convergence estimate where $j \in \mb{N}_0$:
\begin{equation}\label{eq:inexactNK.3}
\|y_{j+1}-y_*\|_Y \leq \left[\frac{L(1+\theta)}{2(1-\sqrt{2 \beta L})}\|y_j - y_*\|_Y  + \theta \frac{1+\sqrt{2 \beta L}}{1-\sqrt{2 \beta L}} \right]\|y_j - y_*\|_Y.
\end{equation}
\end{Theorem}

\begin{Theorem}\label{thm:laxmilgram}[Theorem  11.13 of \cite{Kress1998}]
Let $\mc{H}$ be a Hilbert space and let $\|\cdot\|_{\mc{H}}$ denote a norm on $\mc{H}$ induced by an inner product.  Suppose that $\mc{A}:\mc{H}\rightarrow \mc{H}$ is a bounded linear operator that satisfies the following coercivity estimate:  there exists $c > 0$ so that 
\begin{equation}\label{eq:laxmilgram.1}
\|\mc{A} x\|_{\mc{H}}  \geq c \|x\|_{\mc{H}}, \quad \forall x \in \mc{H}.
\end{equation}
Then $\mc{A}$ is invertible and its inverse is a bounded linear operator.
\end{Theorem}

\subsection{Convergence when the period is known}\label{sec:known}

Consider the following linear DAE:
\begin{equation}\label{eq:lindae2}
\sum_{j=1}^{k+1} T^{-(j-1)} A_j(q(t),T)  u^{(j-1)} = g(t), \quad g \in L^2, \quad q \in \tilde{C}^k, \quad N \in N_0.
\end{equation}
We assume the following for the remainder of Section \ref{sec:known}.
\begin{Assumption}\label{asp:known}
The following two hypotheses are satisfied:
\begin{enumerate}
    \item The DAE \eqref{eq:dae1} satisfies the KP-assumption in Assumption \ref{asp:KPUP}.
    
    \item There exists $\mc{A}, K_L >0$ so that if $q \in \tilde{C}^k$, $N \in \mb{N}_0$, and 
    $$\underset{t\in\mb{R}}{sup}\| P_N A_j(q(t),T) - A_j(p(t),T)\|_{\mc{L}(X)} \leq \mc{A}, \quad j =1,\hdots,k+1$$
    then any solution $v \in \tilde{C}^k$ to \eqref{eq:lindae2} satisfies $\|v\|_{L^2} \leq K_L\|g\|_{L^2}$.
\end{enumerate}

\end{Assumption}
The second statement in Assumption \ref{asp:known} is an isolatedness condition on $p$.  If \eqref{eq:dae1} is a finite dimensional first-order ODE, then the second statement in Assumption \ref{asp:known} is satisfied whenever the linear variational equation associated to $p$ has exponential dichotomy or equivalently, if none of the Floquet exponents of the linear variational equation associated to $p$ is equal to $1$.   The next lemma uses this isolatedness condition to prove invertibility of $\partial_1 R_N(x,T)$ when $N$ is sufficiently large and $x$ is sufficiently close to $\alpha_N$.  We remark (see Remark \ref{rmk:isoremark}) that we could replace the second statement in Assumption \ref{asp:known} with the conclusion of Lemma \ref{lem:DxRinvboundknown} without substantial modification to the statement and proof of Theorem \ref{thm:known}.
\begin{Lemma}\label{lem:DxRinvboundknown}
If $N \geq N_1(\delta(\mc{A}))$ and $x \in \mc{B}(\alpha_N,\delta_N(\mc{A}))$, then $\partial_1 R_{N}(x,T)$ is invertible and its inverse satisfies:
\begin{equation}\label{eq:DxRinvboundknown.1}
\|\partial_1 R_{N}(x,T)^{-1}\| \leq K_{L}.
\end{equation}
\end{Lemma}
\begin{proof}
Lemmas \ref{lem:RisC1}-\ref{lem:DxRlem} imply that $\partial_1 R_N(x,T)$ is a well-defined and bounded linear operator $X^{2N+1}\rightarrow X^{2N+1}$.  Since $X^{2N+1}$ is a Hilbert space, Theorem \ref{thm:laxmilgram} then implies that $\partial_1 R_N(x,T)$ is invertible if it satisfies the following coercivity condition:
\begin{equation}\label{eq:DxRinvboundknown.2}
\|[\partial_1 R_N(x,T)]y\| \geq K_L^{-1} \|y\|, \quad \forall y \in X^{2N+1}.
\end{equation}
We prove Estimate \eqref{eq:DxRinvboundknown.2}.  Let $y \in X^{2N+1}$ and set $z := \partial_1 R_N(x,T) y $. Lemma \ref{lem:DxRlem} implies that:
\begin{equation}\label{eq:DxRinvboundknown.3}
\sum_{j=0}^{k} T^{-(j-1)} P_N A_{N,j}(x,T)  Q_N^{(j-1)}(y)= Q_{N}(z).
\end{equation}
Since $N \geq N_1(\delta(\mc{A}))$ and $x \in \mc{B}(\alpha_N,\delta_N(\mc{A}))$, Estimate \eqref{eq:Alem1.2} of Lemma \ref{lem:Alem1} implies that:
\begin{equation}\label{eq:DxRinvboundknown.4}
\underset{t\in\mb{R}}{\sup}\|P_{N} A_{N,j}(x,T)(t) - A_j(p(t),T)\|_{\mc{L}(X)} \leq \mc{A}, \quad j=1,\hdots,k+1.
\end{equation}
Assumption \ref{asp:known} therefore implies that:
\begin{equation}\label{eq:DxRinvboundknown.5}
\|Q_N(y)\|_{L^2} \leq K_L \|Q_N(z)\|_{L^2}.
\end{equation}
Lemma \ref{lem:Qlem1} and Estimate \eqref{eq:DxRinvboundknown.5} then imply that:
$$\|y\| =  \|Q_N(y)\|_{L^2} \leq K_L\|Q_N(z)\|_{L^2}  = K_L\|z\| = K_L\|[\partial_1 R_N(x,T)]y\|.$$
Thus, Estimate \eqref{eq:DxRinvboundknown.2} is satisfied and $\partial_1 R_N(x,T)$ is invertible.  If $0 \neq \lambda \in X^{2N+1}$, then $Q_N(R_N(x,T)^{-1} \lambda)$ is a well-defined $\tilde{C}^k$ solution to Equation \eqref{eq:DxRinvboundknown.2} with $g = Q_N(\lambda)$ and $q = Q_N(x)$.  Lemma \ref{lem:Qlem1} and Assumption \ref{asp:known} then imply that:
$$\|\partial_1 R_N(x,T)^{-1} \lambda\|  = \|Q_N(\partial_1 R_N(x,T)^{-1} \lambda)\|_{L^2} \leq K_L\|Q_N(\lambda)\|_{L^2} = K_L\|\lambda\|.$$
Therefore $\|\partial_1 R_N(x,T)^{-1}\| \leq K_L$.
\end{proof}
For the remainder of Section \ref{sec:known} we use the following abbreviated notation:
\begin{equation*}
\delta = \delta(\mc{A}), \quad \delta_N = \delta_N(\mc{A}), \quad  N_1 = N_1(\delta(\mc{A})), \quad R_N(\cdot,T) = R_N(\cdot).
\end{equation*}
The following theorem establishes the convergence rate of the HB method in terms of $N$ when the period $T$ is known exactly.  In summary, it shows that if $N$ is sufficiently large and $K_p(N)$ converges fast enough relative to $M_1$ and $\rho_1$, then the HB residual function $R_N(\cdot)$ has a unique zero $x_*^N$ nearby $\alpha_N$ where (i) any inexact Newton process whose relative residual satisfies a sufficiently small tolerance converges to $x_*^N$ at the expected rate when the initial guess $x_0^N$ is sufficiently close to $\alpha_N$; (ii)  $Q_N(x_*^N)$ converges to $p$ in $H^k$ with convergence rate determined by $K_p(N)$ and structure of the DAE via $M_1$ and $\rho_1$.
\begin{Theorem}\label{thm:known}
Fix $\ve \in (0,1/2)$ and assume that:
\begin{equation}\label{eq:known.0}
K_p(N) = \mathcal{O}((N+1)^{-\sigma}), \quad \sigma > 2(M_1 + \rho_1).
\end{equation}
  There exists $N_2 \geq N_1$ and constants $C, C_0,\hdots,C_k > 0$ so that if $N \geq N_2$ and $\xi_N > 0$ is such that 
$$\xi_N \leq \text{min}\{\delta_N,C(N+1)^{-(M_1+\rho_1)}(1+2\pi N + \hdots (2\pi N)^k)^{-1}\},$$
then there exists a unique zero $x_*^N$ of $R_N(\cdot,T)$ in $\mc{B}(\alpha_N,\xi_N)$ and the following hold:
\begin{enumerate}
    \item The following estimates hold:
    \begin{equation}\label{eq:known.1}
\|p^{(j)} - Q_N^{(j)}(x_*^N)\|_{L^2} \leq C_j (N+1)^{M_1 + \rho_1 + j} K_p(N), \quad j =0,\hdots,k.
\end{equation}

\item  An inexact Newton process  (defined as in \eqref{eq:inexactNK.1}) for computing a zero of $R_N(\cdot)$ with relative residual tolerance $\theta_N$ satisfying $\theta_N \leq (1-\sqrt{2\ve})/(1+\sqrt{2\ve})$ and initial condition $x_0^N \in \mc{B}(\alpha_N,\xi_N)$ results in a well-defined sequence $\{x_j^N\}_{j=0}^{\infty}$ with $\underset{j \rightarrow \infty}{\lim}\|x_j^N-x_*^N\| =0$ that satisfies the following estimate:
\begin{equation}\label{eq:known.2}
\|x_{j+1}^N-x_*^N\|\leq \hat{C}_{N,1} \|x_j^N-x_*^N\|^2 + \theta_N \hat{C}_{N,2}\|x_j^N-x_*^N\|
\end{equation}
for constants $\hat{C}_{N,1},\hat{C}_{N,2} > 0$ that may depend on $N$.
\end{enumerate}

\end{Theorem}
\begin{proof}
The proof is an application of Theorem \ref{thm:inexactNK}.  Assume that $N \geq N_1$ and $x_0^N \in \mc{B}(\alpha_N,\xi_N)$ for some $\xi_N \leq \delta_N$.  Since $\xi_N \leq \delta_N$, Lemma \ref{lem:DxRinvboundknown} implies that $\partial_1 R_N(x_0^N)$ is invertible with:
$$\|\partial_1 R_N(x_0^N)^{-1}\| \leq K_L.$$  
We next estimate $\|\partial_1 R_N(x_0^N)^{-1} ( \partial_1 R_N(z) - \partial_1 R_N(y))\|$. If $y,z \in \mc{B}(\alpha_N,\delta_N)$, then since \eqref{eq:dae1} satisfies the KP-assumption and with $L_1$ as defined as in \eqref{eq:DxxRlem.2a}:
\begin{align*}
\| & \partial_1 R_N(x_0^N)^{-1} (\partial_1 R_N(z) - \partial_1 R_N(y))\| \\
& \leq \|\partial_1 R_N(x_0^N)^{-1}\| \cdot \|\partial_1 R_N(z) - \partial_1 R_N(y)\|\\
& \leq K_L \|\partial_1 R_N(z) - \partial_1 R_N(y)\| && (\|\partial_1 R_N(x_0^N)^{-1}\| \leq K_L)\\
& \leq K_L L_1 N^{M_1}\|z-y\| && \text{(Lemma \ref{lem:DxxRlem}, Estimate \eqref{eq:DxxRlem.2a})}\\
& \leq \underbrace{\text{max}\{K_L L_1, 2\delta^{-1} (2\pi)^{k} \}}_{:= \hat L}(N+1)^{M_1 + \rho_1}\|z-y\|.
\end{align*}
We define $L_N > 0$ as:
\begin{equation}\label{eq:Ldef}
L_N = \hat{L}(N+1)^{M_1 +\rho_1}.
\end{equation} 
By construction $L_N^{-1} \leq \delta_N$.  Therefore $\mc{B}(\alpha_N,L_N^{-1}) \subseteq \mc{B}(\alpha_N,\delta_N)$ so:
\begin{equation}\label{eq:known.NK1}
\|\partial_1 R_N(x_0^N)^{-1}(\partial_1 R_N(z) - \partial_1 R_N(y))\| \leq L_N \|z-y\|,  \quad \forall y,z \in \mc{B}(\alpha_N,L_N^{-1}).
\end{equation}
We next estimate $\|\partial_1 R_N(x_0^N)^{-1}R_N(x_0^N)\|$:
\begin{align*}
\|\partial_1 R_N(x_0^N)^{-1}R_N(x_0^N)  \| & \leq \|\partial_1 R_N(x_0^N)^{-1}\| \cdot \|R_N(x_0^N)\| \\
& \leq K_L \|R_N(x_0^N)\| && (\|\partial_1 R_N(x_0^N)^{-1}\|\leq K_L) \\
& \leq K_L \|\mc{F}_N(x_0^N,T)\|_{L^2}. && \text{(Lemma \ref{lem:DxxRlem}, Estimate \eqref{eq:DxxRlem.1})}
\end{align*}
Since $\|x_0^N - \alpha_N\| \leq \xi_N \leq \delta_N \leq \delta(\mc{A}) \leq 1$ and $N \geq N_1$, Lemma \ref{lem:Qlem2} implies: 
$$\|Q_N(x_0^N) -p\|_{H^k} \leq \delta(\mc{A}) \leq 1.$$
Estimate \eqref{eq:defJ1J2} with $\tau = T$ then implies that:
$$\|\mc{F}_N(x_0^N,T)\|_{L^2} \leq J_1\|Q_N(x_0^N) - p\|_{H^k}.$$
Therefore:
$$\|\partial_1 R_N(x_0^N)^{-1}R_N(x_0^N)\|  \leq  K_L J_1 \|Q_N(x_0^N) - p\|_{H^k}.$$
This together with Estimate \eqref{eq:Qpbound} implies:
\begin{equation}\label{eq:known.NK2}
\|\partial_1  R_N(x_0^N)^{-1}R_N(x_0^N)\|  \leq K_L J_1\Bigr(\sum_{j=0}^{k} (2\pi N)^j\Bigr)\|x_0^N - \alpha_N\|+ K_L J_{1} K_p(N) \equiv \beta_N(x_0^N). 
\end{equation}
Then 
$$\beta_N  (x_0^N) L_N \leq  K_L J_{1} L_N \sum_{j=0}^{k}(2\pi N)^j \|x_0^N - \alpha_N\| + K_L J_{1} L_N K_p(N).$$
Estimate \eqref{eq:known.0} implies there exists $N_* \geq N_1$ such that if $N \geq N_*$, then:
$$ K_L J_{1} L_N = K_L J_{1}\hat L  (N+1)^{M_1+ \rho_1}K_p(N) \leq \ve/2.$$
We let $C > 0$ be any constant such that:
\begin{equation}\label{eq:known.Cdef}
C \leq \text{min}\{\ve / (2 K_L J_1 \hat L ),\hat L^{-1}/2\}.
\end{equation}
So, if
\begin{equation}\label{eq:known.asp}
N \geq N_*, \quad \xi_N \leq \text{min}\{\delta_N,C(N+1)^{-(M_1+\rho_1)}(1+2\pi N + \hdots + (2\pi N)^k)^{-1}\},
\end{equation}
then 
\begin{equation}\label{eq:known.xnltnl}
\beta_N(x_0^N) L_N \leq \ve < 1/2, \quad \xi_N \leq L_N^{-1}/2.
\end{equation}
For the remainder of the proof, assume $N$ and $\xi_N$ satisfy the estimates in \eqref{eq:known.asp}.  Suppose that we approximate a zero of $R_N(\cdot)$ using an inexact Newton method with initial guess $x_0^N \in \mc{B}(\alpha_N,\xi_N)$ and relative residual tolerance $\theta_N$ satisfying:
$$\theta_N\leq (1-\sqrt{2 \ve})/(1+\sqrt{2 \ve}) \leq (1-\sqrt{2 \beta_N(x_0^N) L_N})/(1+\sqrt{2\beta_N(x_0^N) L_N}).$$
As shown above, Estimates \eqref{eq:known.NK1}-\eqref{eq:known.NK2} are satisfied with $\beta_N(x_0^N) L_N \leq \ve < 1/2$.  Theorem \ref{thm:inexactNK} then implies that the sequence $\{x_j^N\}_{j=0}^{\infty}$ generated from this inexact Newton procedure is well-defined and there exists $x_*^N$ which is the unique zero of $R_N(\cdot)$ in $\mc{B}(x_0^N,L_N^{-1})$ and the following estimates hold for all $j \in \mb{N}_0$:
$$\|x_*^N-x_0^N\|\leq (1-\sqrt{1-2\beta_N(x_0^N) L_N})/L_N \leq 2 \beta_N(x_0^N),$$
\begin{align*}\|& x_{j+1}^N -x_*^N\| \\
& \leq \underbrace{\left[ \frac{(1+\theta_N)L_N}{2(1-\sqrt{2 \beta_N(x_0^N) L_N})}\right]}_{:= \hat C_{N,1}}\|x_{j}^N-x_*^N\|^2 + \theta_N\underbrace{\left[\frac{(1+\sqrt{2\beta_N(x_0^N) L_N})}{1-\sqrt{2 \beta_N(x_0^N) L_N}} \right]}_{:= \hat C_{N,2}}\|x_j^N - x_*^N\|.
\end{align*}
This proves Estimate \eqref{eq:known.2}.  Repeating the above process with $\alpha_N$ as the initial guess in the inexact Newton process implies that there exists a zero $y_*^N$ of $R_N(\cdot)$ with 
\begin{equation*}
\|\alpha_N - y_*^N\| \leq 2 \beta_N(\alpha_N) = 2 K_L J_1 \hat L (N+1)^{M_1 + \rho_1} K_p(N).
\end{equation*}
Estimate \eqref{eq:known.0} implies that there exists $N_2 \geq N_*$ so that if $N \geq N_2$, then:
\begin{equation}\label{eq:known.yNdef}
2 K_L J_1 \hat L (N+1)^{M_1 + \rho_1} K_p(N) \leq \hat L^{-1} (N+1)^{-M_1 -\rho_1}/2 = L_N^{-1}/2.
\end{equation}
So, if $N \geq N_2$, then:
\begin{align*}
\|y_*^N  - x_0^N\| & \leq \|y_*^N - \alpha_N\| + \|\alpha_N - x_0^N\| && \text{(Triangle inequality)} \\
& \leq L_N^{-1}/2 + \|\alpha_N - x_0^N\| && \text{(Estimate \eqref{eq:known.yNdef})}\\
& \leq  L_N^{-1}/2 + \xi_N && \text{($\|\alpha_N - x_0^N\| \leq \xi_N$)} \\
& \leq L_N^{-1}/2 + L_N^{-1}/2 = L_N^{-1} && \text{(Estimate \eqref{eq:known.xnltnl})}
\end{align*}
Therefore $y_*^N \in \mc{B}(x_0^N,L_N^{-1})$ whenever $N \geq N_2$.  So, if $N \geq N_2 \geq N_*$, then the facts that $R_N(y_*^N) = 0$ and $x_*^N$ is the unique zero of $R_N$ in $\mc{B}(x_0^N,L_N^{-1})$ whenever $N \geq N_*$ implies that $y_*^N = x_*^N$ and therefore:
\begin{equation}\label{eq:known.last}
\|x_*^N - \alpha_N \| = \|y_*^N - \alpha_N\| \leq 2 K_L J_1 \hat L (N+1)^{M_1 + \rho_1} K_p(N), \quad N \geq N_2.
\end{equation}
We then have the following estimate:
\begin{align*}
\|p^{(j)} & - Q_N^{(j)}(x_*^N)\|_{L^2}  \\
&\leq \|p^{(j)} - Q_N^{(j)}(\alpha_N)\|_{L^2} + \|Q_N^{(j)}(\alpha_N) - Q_N^{(j)}(x_*^N)\|_{L^2} && \text{(Triangle inequality)} \\
& \leq  K_p(N) + \|Q_N^{(j)}(\alpha_N) - Q_N^{(j)}(x_*^N)\|_{L^2} && \text{(Estimate \eqref{eq:pfourierbound})} \\
& \leq K_p(N) + (2\pi N)^j\|x_*^N-\alpha_N\| && \text{(Lemma \ref{lem:Qlem1})}\\
& \leq  (1 + 2 K_L J_1 \hat L  (2\pi N)^j (N+1)^{M_1 + \rho_1})K_p(N) && \text{(Estimate \eqref{eq:known.last})} \\
& \leq (1 + 2 K_L J_1 \hat L (2\pi)^{j}) (N+1)^{M_1 + \rho_1 + j} K_p(N).
 \end{align*}
This proves Estimate \eqref{eq:known.1} with $C_j = 1 + 2 K_L J_1 \hat L (2\pi)^{j}$  for $j = 0,\hdots,k$.  This completes the proof.
\end{proof}

\subsection{Convergence when the period is unknown}\label{sec:unknown}


In this section we analyze the convergence of the HB method when the period $T$ is unknown.  Our approach, similar to that taken in \cite{Shinohara1981}, studies convergence using an appropriately defined phase condition.  We define a phase condition $\gamma:(L^2)^{k+1} \rightarrow \mb{R}$ such that:
\begin{equation}\label{eq:defgamma}
\gamma \in C^1((L^2)^{k+1},\mb{R}), \quad \gamma(p,p^{(1)},\hdots,p^{(k)}) = 0.
\end{equation}
Let $D \gamma$ denote the Frech\'et derivative of $\gamma$.  We use the following notation where $q \in \tilde{C}^k$, $N \in \mb{N}_0$, and $x \in X^{2N+1}$:
\begin{equation}\label{eq:sigmanotation}
\begin{array}{c} \sigma(q) = \gamma(q,q^{(1)},\hdots,q^{(k)}), \quad D\sigma(q) = D\gamma(q,q^{(1)},\hdots,q^{(k)}),\\
\sigma_N(x) = \sigma(Q_N(x)).
\end{array}
\end{equation}
The following constants are well-defined due to \eqref{eq:defgamma}:
\begin{itemize}
    \item Define $J_{\sigma} > 0$ so that if $q \in \tilde{C}^k$ is such that $\|q-p\|_{H^k} \leq 1$, then:
    \begin{equation}\label{eq:defJsigma}
        \|\sigma(q)\|_{L^2} \leq J_{\sigma}\|q-p\|_{H^k}.
    \end{equation}
    
  \item Let $\delta_{\sigma}(\ve)$ be the greatest number in  $(0,\delta(\ve)]$ such that if $q \in \tilde{C}^k$ is such that $\|q - p\|_{H^k} \leq \delta_{\sigma}(\ve)$, then:
       \begin{equation}\label{eq:Dsigbound}
        \|D \sigma(q) - D \sigma(p) \|_{L^2} \leq \ve.
       \end{equation}

      \item Let $\delta_{N,\sigma}(\ve) := \zeta_N(\delta_{\sigma}(\ve))$.
\end{itemize}
We approximate $p$ and $T$ by computing zeros of $\hat R_N:X^{2N+1} \times (0,\infty) \rightarrow X^{2N+1} \times \mb{R}$ where $\hat R_N$ is defined by:
\begin{equation}\label{eq:defRhat}
\hat R_N(x,\tau) = (R_N(x,\tau),\sigma_N(x)).    
\end{equation}
Lemma \ref{lem:RisC1} and the fact that $\gamma \in C^1((L^2)^{k+1},\mb{R})$ implies that:
\begin{equation}\label{eq:RhatisC1}
\hat R_N \in C^1(X^{2N+1} \times (0,\infty),X^{2N+1}\times \mb{R}), \quad N \in \mb{N}_0.
\end{equation}
We define a linear map $\Gamma:\mc{L}((L^2)^{k+1},\mb{R}) \times \tilde{C}^k \rightarrow X \times  \mb{R}$ by: \begin{equation}\label{eq:Gammadef}
\Gamma(\rho,v) = (v(1)-v(0), \rho \cdot (v,v^{(1)},\hdots,v^{(k)}))
\end{equation}
We make the following assumption for the remainder of Section \ref{sec:unknown}.
\begin{Assumption}\label{asp:unknown}
The following hold:
\begin{enumerate}
\item The DAE \eqref{eq:dae1} satisfies the UP-assumption. 

\item Consider the following boundary value problem (BVP) where $g \in L^2$ and $\beta \in \mb{R}$:
\begin{equation}\label{eq:unknownbvp}
\left\{
\begin{array}{lcr}
(\sum_{j=1}^{k+1} \tau^{-(j-1)} A_j(q,\tau) u^{(j-1)} +  A_{k+2}(q,\tau) \eta,\eta^{(1)}) = (g,0)\\
\Gamma(D \sigma(q),u) = (0,\beta)
\end{array}
\right.
\end{equation}
Assume there exist $\mc{A}, K_{L}> 0$ so that if $q \in \tilde{C}^k$ and $\tau >0$ are such that:
$$\|D\sigma(q) - D\sigma(p)\|_{L^2} \leq \mc{A}, \quad \|P_N A_{k+2}(q,\tau)- A_{k+2}(p,T)\|_{L^2} \leq \mc{A},$$
$$\|P_N A_j(q,\tau) - A_j(p,T)\|_{L^2} \leq \mc{A}, \quad j =1,\hdots,k+1,$$
then any solution $t \mapsto (v(t),\xi_0)$ to \eqref{eq:unknownbvp}, where $v \in \tilde{C}^k$ and $\xi_0 \in \mb{R}$, satisfies
\begin{equation}\label{eq:konwnbvpest}
\left(\|v\|_{L^2}^2 + |\xi_0|^2\right)^{1/2} \leq K_L(\|g\|_{L^2}^2 + |\beta|^2)^{1/2}.
\end{equation}

\item There exists $L_{\sigma,1},\hdots,L_{\sigma,k+1} \geq 0$ so that if $q_1,q_2 \in \tilde{C}^{k}$ are such that $\|q_l - p\|_{H^k} \leq 1$ for $l=1,2$, then:
\begin{align}
\label{eq:defLsig} \|D\sigma(q_1)\cdot (q_1,q_1^{(1)},\hdots,q_1^{(k)})-D\sigma(q_2) & \cdot (q_2,q_2^{(1)},\hdots,q_2^{(k)})\|_{L^2} \\
& \nonumber \leq \sum_{j=1}^{k+1} L_{\sigma,j} \|q_1^{(j-1)} - q_2^{(j-1)}\|_{L^2}.
\end{align}
\end{enumerate}
\end{Assumption}
The second statement in Assumption \ref{asp:unknown} is an isolatedness assumption.  For example, if \eqref{eq:dae1} is a finite dimensional first-order ODE, then this statement will hold if $1$ is a simple eigenvalue of the monodromy matrix associated to $p$ and the phase condition satisfies a transversality condition.  The following lemma uses Assumption \ref{asp:unknown} to give sufficient conditions for invertibility of $D \hat R_N(x,\tau)$ and provides a bound on the norm of its inverse.
\begin{Lemma}\label{lem:DxRhatlem}
Assume that $N \geq N_1(\delta(\mc{A}))$, $x \in \mc{B}(\alpha_N,\delta_{N,\sigma}(\mc{A}))$, and $|\tau - T| \leq \delta(\mc{A})$.  Then $D \hat R_N(x,\tau)$ is invertible and 
\begin{equation}\label{eq:DxRhat.1}
\|[D \hat R_N(x,\tau)]^{-1}\| \leq K_L.
\end{equation}
\end{Lemma}
\begin{proof}
Lemmas \ref{lem:RisC1}-\ref{lem:DxRlem} imply that $D \hat R_N(x,\tau)$ is a well-defined and bounded linear operator $X^{2N+1}\times \mb{R}\rightarrow X^{2N+1}\times \mb{R}$. Theorem \ref{thm:laxmilgram} then implies that $D \hat R_N(x,\tau)$ is invertible if it satisfies the following coercivity condition:
\begin{equation}\label{eq:DxRinvboundunknown.2}
\|[D \hat R_N(x,\tau)](y,s)\| \geq K_L^{-1} \|(y,s)\|, \quad \forall (y,s) \in X^{2N+1} \times \mb{R}.
\end{equation}
For simplicity of notation, we write:
\begin{equation*}
\begin{array}{c}\hat R_N(x,\tau) = \hat R_N, \quad D \hat R_N(x,\tau) = D \hat R_N, \quad \sigma_N(x) = \sigma_N,  \\
A_{N,j}(x,\tau) = A_{N,j}, \quad j = 1,\hdots,k+2.
\end{array}
\end{equation*}
The definition of $\sigma_N$ and Lemma \ref{lem:Qlem1} imply that the following holds for all $y \in X^{2N+1}$:
\begin{equation}\label{eq:DxRhat.2}
\left(\frac{\partial}{\partial a} \Bigr|_{a=x} \sigma_N(a) \right)y = D\sigma(Q_N(x)) \cdot (Q_N(y),Q_N^{(1)}(y),\hdots,Q_N^{(k)}(y)).
\end{equation}
Lemmas \ref{lem:RisC1}-\ref{lem:DxRlem}, Equation \eqref{eq:DxRhat.2}, and the definition of $\hat R_N$ imply that $D \hat R_N \cdot (y,s) = (z,w)$ where $(y,s),(z,w)\in X^{2N+1} \times \mb{R}$ if and only if the function defined by $t \mapsto (Q_N(y)(t),s)$ is a solution to the following boundary value problem:
\begin{equation}\label{eq:DxRhat.3}
\left\{
\begin{array}{lcr}
(\sum_{j=1}^{k+1} \tau^{-(j-1)} P_N A_{N,j} u^{(j-1)} +  P_N A_{N,k+2} \eta,\eta^{(1)}) = (Q_N(z),0) \\
\Gamma(D \sigma_N(x),(u,\eta)) = (0,w). \\
\end{array}
\right.
\end{equation}
Since $\delta_{N,\sigma}(\mc{A}) = \zeta_N(\delta_{\sigma}(\mc{A}))$, Lemma \ref{lem:Qlem2} implies that $\|Q_N(x) - p\|_{H^k} \leq \delta_{\sigma}(\mc{A})$ and therefore Estimate \eqref{eq:Dsigbound} implies that:
\begin{equation}\label{eq:DxRhat.4}
\|D\sigma(Q_N(x))- D\sigma(p)\|_{L^2} \leq \mc{A}.
\end{equation}
Since $\delta_{N,\sigma}(\mc{A}) \leq \delta_N(\mc{A})$, Lemma \ref{lem:Alem1} implies that: 
\begin{equation}\label{eq:DxRhat.5}
\begin{array}{c}
\underset{t \in [0,1]}{\sup}\|P_N A_{N,j}- A_j(p,T)\|_{\mc{L}(X)} \leq \mc{A}, \quad j =1,\hdots,k+1, \\
\underset{t \in [0,1]}{\sup}\|P_N A_{k+2} - A_{k+2}(p,T)\|_{\mc{L}(\mb{R},X)} \leq \mc{A}.
\end{array}
\end{equation}
We then have:
\begin{align*}
K_L^2\|D\hat R_N \cdot (y,s)\|^2 & = K_L^2\|(z,w)\|^2\\
&= K_L^2(\|Q_N(z)\|_{L^2}^2 + |w|^2) && \text{(Lemma \ref{lem:Qlem1})} \\
& \geq  \|Q_N(y)\|^2_{L^2} + |s|^2 && \text{(Assumption \ref{asp:unknown})}\\
   & = \|y\|^2 + |s|^2 = \|(y,s)\|^2.   && \text{(Lemma \ref{lem:Qlem1})} 
\end{align*}
Therefore $D\hat R_N$ satisfies the coercivity bound \eqref{eq:DxRinvboundknown.2} and is therefore invertible.  Similar estimates as above show that $\|[D\hat R_N]^{-1}\| \leq K_L$.
\end{proof}
Based on Assumption \ref{asp:unknown}, we define $\tilde{M}_2 \in \mb{N}_0$ by
\begin{equation}\label{eq:deftildem2}
\tilde{M}_2 = \text{max}(\{M_2\} \cup \{j-1 : j \in \{1,\hdots,k+1\} \text{ and } L_{j,\sigma} > 0\})
\end{equation}
and define $\rho_2$ as the minimal element of $\mb{N}_0$ so that the following is satisfied:
\begin{equation}\label{eq:rho2def}
1 + (2\pi N) + \hdots (2\pi N)^k \leq (2\pi)^k(N+1)^{\tilde{M}_2 + \rho_2}, \quad N \in \mb{N}_0.
\end{equation}

The following lemma is a consequence of Lemma \ref{lem:DxxRlem}, Assumption \ref{asp:unknown}, and the definitions of $\tilde{M}_2$ \eqref{eq:deftildem2}.
\begin{Lemma}\label{lem:DxxRhatlem}
If $\ve \in (0,1]$, $N \geq N_1(\delta(\ve))$, then the following holds for all $y,z \in \mc{B}(\alpha_N,\zeta_N(\ve))$ and $s,w \in \{\tau:|\tau-T| \leq T/2\}$:
\begin{equation}\label{eq:DxxRhatlem.1}
\|D \hat R_N(y,s)- D \hat R_N(z,w)\| \leq \hat L_2 N^{\tilde{M}_2} \|(y,s)-(z,w)\|
\end{equation}
where $\hat L_2 = L_2 + \sum_{j=1}^{k+1} L_{\sigma,j}(2\pi)^{j-1}$ and $L_2$ is defined by \eqref{eq:defm2}.
\end{Lemma}
The proof of the following theorem, which is omitted, follows the same steps as that of Theorem \ref{thm:known}: applying Theorem \ref{thm:inexactNK}, but using Estimate \ref{eq:DxxRhatlem.1} of Lemma \ref{lem:DxxRhatlem} rather than Estimate \eqref{eq:DxxRlem.2a} from Lemma \ref{lem:DxxRlem} and Lemma \ref{lem:DxRhatlem} rather than Lemma \ref{lem:DxRlem}.  Its conclusions are essentially the same as those of Theorem \ref{thm:known}, but now includes estimates for bounds on the approximate period as well as the approximate solution.
\begin{Theorem}\label{thm:unknown}
Fix $\ve \in (0,1/2)$ and assume that:
\begin{equation}\label{eq:unknown.0}
K_p(N) = \mathcal{O}((N+1)^{-\sigma}), \quad \sigma > 2(\tilde{M}_2 + \rho_2).
\end{equation}
  There exists $N_3 \geq N_1$ and constants $D,D_0,\hdots,D_{k+1} > 0$  so that if $N \geq N_3$ and $\xi_N > 0$ is such that: 
$$\xi_N \leq \text{min}\{\delta_{N,\sigma}(\mc{A}),D(N+1)^{-(\tilde{M}_2+\rho_2)}(1+2\pi N + \hdots (2\pi N)^k)^{-1}\},$$
then there exists a unique zero $(x_*^N,\tau_*^N)$ of $\hat R_N$ in $\mc{B}((\alpha_N,T),\xi_N)$ and the following hold:
\begin{enumerate}
    \item  The following estimates hold:
    \begin{equation}\label{eq:unknown.1a}
\|p^{(j)} - Q_N^{(j)}(x_*^N)\|_{L^2} \leq D_j (N+1)^{\tilde{M}_2 + \rho_2 + j} K_p(N), \quad j =0,\hdots,k,
\end{equation}
\begin{equation}\label{eq:unknown.1b}
|\tau_*^N-T| \leq D_{k+1} (N+1)^{\tilde{M}_2+\rho_2} K_p(N).
\end{equation}

\item  The inexact Newton iteration for computing a zero of $R_N$ with relative residual tolerance $\theta_N$ satisfying $\theta_N \leq (1-\sqrt{2\ve})/(1+\sqrt{2\ve})$ and initial condition $(x_0^N,\tau_0^N) \in \mc{B}((\alpha_N,T),\xi_N)$ results in a well-defined sequence $\{(x_j^N,\tau_k^N)\}_{j=0}^{\infty}$ with $\underset{j \rightarrow \infty}{\lim}\|(x_j^N,\tau_j^N)-(x_*^N,\tau_*^N)\| =0$ that satisfies the following estimate:
\begin{align}
\|(x_{j+1}^N &,\tau_{j+1}^N)-(x_*^N,\tau_*^N)\|  \leq  \nonumber \\
&  \hat{D}_{N,1}\|(x_j^N,\tau_j^N)-(x_*^N,\tau_*^N)\|^2 + \theta_N \hat{D}_{N,2}\|(x_j^N,\tau_j^N)-(x_*^N,\tau_*^N)\|  \label{eq:unknown.2}
\end{align}
for constants $\hat{D}_{N,1},\hat{D}_{N,2} > 0$ that may depend on $N$.
\end{enumerate}
\end{Theorem}

We close this our discussion of convergence with a remark discussing our motivations for using the isolatedness conditions (Assumptions \ref{asp:known}-\ref{asp:unknown}) in our analysis.
\begin{Remark}\label{rmk:isoremark}
To employ a Newton-Kantorovich type argument to prove existence and uniqueness of a zero of $R_N$ (resp. $\hat R_N$) as was done in Theorem \ref{thm:known} (resp. Theorem \ref{thm:unknown}) requires invertibility of the linearized operator $\partial_1 R_N$ (resp. $D \hat R_N$).  We chose to employ isolatedness conditions (Assumptions \ref{asp:known}-\ref{asp:unknown}) to prove invertibility based on the traditional approaches in \cite{Urabe1965} and \cite{Shinohara1981}.  A more abstract approach would simply be to assume that $\partial_1 R_N(x,T)$ (resp. $D\hat R_N(x,\tau)$) is invertible for all $x$ (resp. $(x,\tau)$) sufficiently near to $\alpha_N$ (resp. ($\alpha_N,T)$).  Our choice was motivated by the observation that invertibility of these linear operators (via Lemma \ref{lem:DxRlem}) and the isolatedness assumptions are both related to the nonautonomous linear DAE obtained by linearizing \eqref{eq:dae2} at $p$.  
\end{Remark}

\section{Examples}\label{sec:examples}

In this section we present several numerical examples arising in structural mechanics and circuit modeling to illustrate the theory developed in Section \ref{sec:convergence}.  Typically, an exact expression for $p$ is not available which means the error estimates from Theorems \ref{thm:known}-\ref{thm:unknown} cannot be applied directly.  We therefore take an indirect approach.  For $N \in \mb{N}_0$, we let $\hat \alpha_N$ denote the computed value of $\alpha_N$ from some inexact Newton process and let $\hat T_N$ denote the computed or exact period when $T$ is unknown or known, respectively.  Set $\hat p_N = Q_N(\hat \alpha_N)$ and define the error $E(N)$ by:
\begin{equation}
E(N) = \|\mc{F}_N(\hat \alpha_N, \hat T_N)\|_{L^2}.
\end{equation}
The smoothness assumptions on $F$ imply that the following estimate holds:
\begin{equation}\label{eq:daeresacc}
E(N) = \mc{O}\left(\|\hat p_N - p_N\|_{H^k} + K_p(N)  + |\hat T_N - T|\right).
\end{equation}
The conclusions of Theorems \ref{thm:known}-\ref{thm:unknown} can then be used to determine asymptotic bounds on $E(N)$ whenever an isolatedness condition (Assumption \ref{asp:known} or \ref{asp:unknown}) is satisfied.  In particular, if $F \in C^{\infty}(X^k \times \mb{R},X)$, then $K_p(N) = \mc{O}(e^{-\kappa_p N})$ for some $\kappa_p > 0$.  By Theorems \ref{thm:known}-\ref{thm:unknown} we would then expect that $E(N) = \mc{O}(e^{-\kappa N})$ where $0 < \kappa \leq \kappa_p$.  We take $g_1 = g_2 = I_X$ throughout this section.  We next empirically verify such exponential bounds on $E(N)$ for several example problems.

\subsection{3D nonautonomous DAE}\label{sec:3ddae}

We consider a variant of the nonlinear circuit network of Example 2 in \cite{lamouretal1998} (see also \cite{barzsuschke1994}).  Let $k=1$ and $X = \mb{R}^3$ with norm given by the standard Euclidean $2$-norm.  Define $G \in C^{\infty}(X^{k+1} \times \mb{R},X)$ by:
\begin{equation}\label{eq:3dtoy}
G(u,u^{(1)} ,t) = \left[\begin{array}{c}
u_1^{(1)} + u_1 - (e^{-u_1 - u_3} -1) \\
u_2^{(1)} + u_2+ u_3 + \sin(2\pi t)\\
u_2 + u_3 +\sin(2\pi t) + e^{u_3} - e^{-u_1-u_3}
\end{array}\right], \quad u = (u_1,u_2,u_3).
\end{equation}
The DAE \eqref{eq:dae2} with $G$ defined by \eqref{eq:3dtoy} and $g_1 = g_2 = I_X$ possesses an isolated periodic solution with $T = 1$ (see e.g. \cite{lamouretal1998}).  Due to Theorem \ref{thm:known} and the fact that $G \in C^{\infty}(X^{k+1} \times \mb{R},X^{k+1})$,  we expect that $E(N) = \mc{O}(e^{-\kappa N})$ for some $\kappa > 0$. 

We now discuss the the inexact Newton method used to compute a zero of $R_N(\cdot,1)$ and approximate the $1$-periodic solution to \eqref{eq:3dtoy}.  The Fourier integrals defining $R_N$ and the integral defining $E(N)$ are approximated with an absolute and relative error tolerance of $10^{-15}$.  The Newton iterates are computed via direct linear solves using finite difference approximations for derivatives.  A relative and absolute error tolerance of $10^{-15}$ is used as the stopping criteria for the Newton iteration.  The initial guess is constructed in two ways resulting in two experiments. In the first experiment the initial guess is the zero vector in $X^{2N+1} = \mb{R}^{3(2N+1)}$ for all $N \geq 2$.  In the second experiment the initial guess is taken to be the zero vector in $X^{2N+1} = \mb{R}^{3(2N+1)}$ when $N = 2$ and then for $N = 4,6,8,\hdots$ the initial guess is the previously computed solution to $R_{N-2}(\cdot,T)$, extended to a vector in $X^{2N+1}$ by appending zeros.  Note that we do not verify if the initial guesses satisfy the hypotheses of Theorem \ref{thm:known}, relying on the conventional wisdom that the Newton-Kantorovich convergence estimates are overly pessimistic.

 In Figures \ref{fig:3dtoy.1}-\ref{fig:3dtoy.2}, we show the results of our two experiments for $N = 2l$ where $1 \leq l \leq 8$.  It is apparent that $E(N)$ is nearly identical in both experiments (due to the uniqueness of the zero of $R_N(\cdot,T)$) and $E(N) = \mc{O}(e^{-\kappa N})$ where $\kappa \approx 2$.  In the first experiment the number of Newton iterations needed to satisfy the error tolerances remains fixed as $N$ increase.  The method of forming the initial guess in the second experiment results in a decreasing number of Newton iterations required as $N$ increases.  While the approximate solutions produced in both of these experiments are essentially identical, the method of forming the initial guess used in the second experiment results in a more computationally efficient method.

\begin{figure}
\includegraphics[scale=0.4]{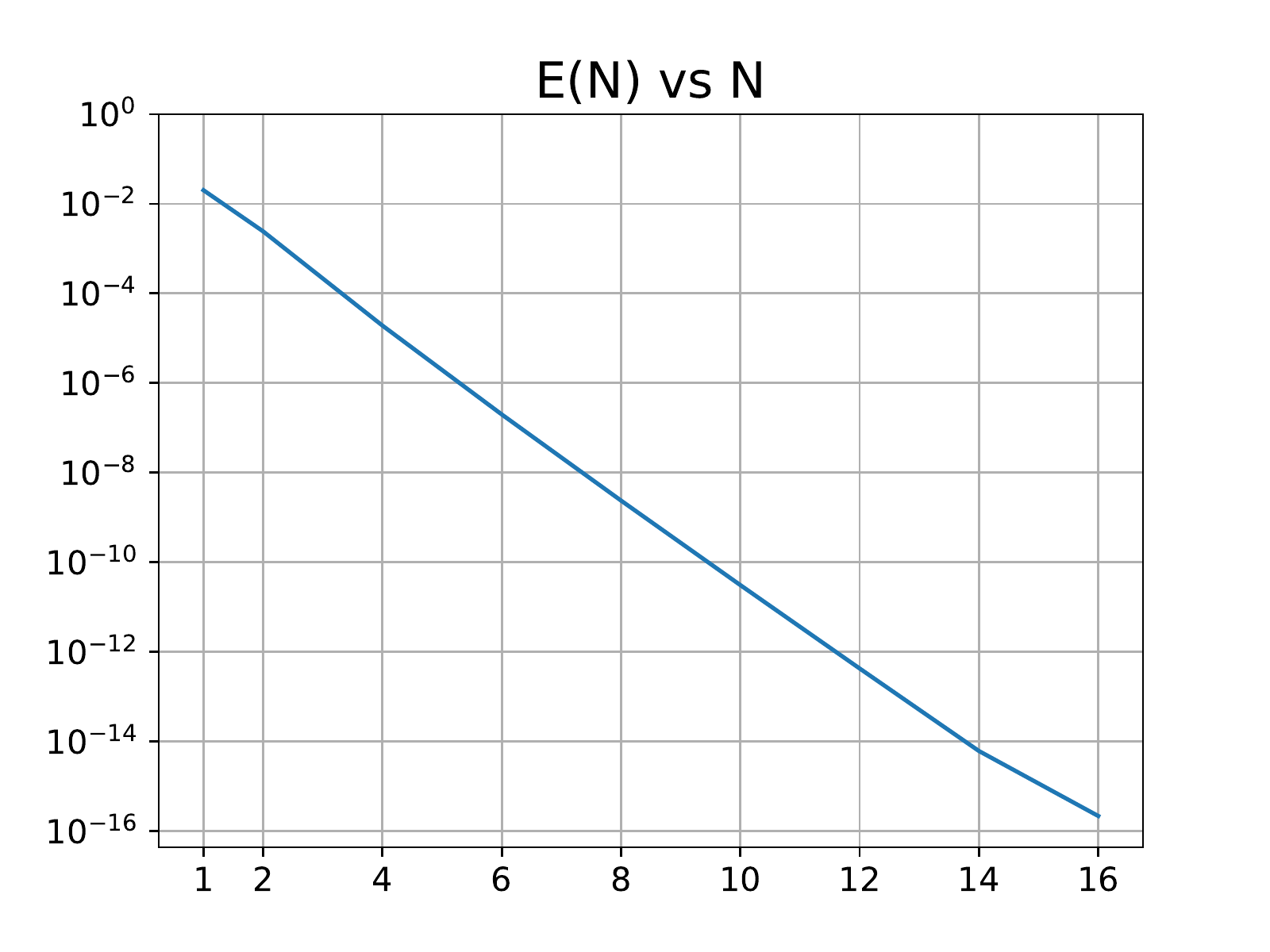}\includegraphics[scale=0.4]{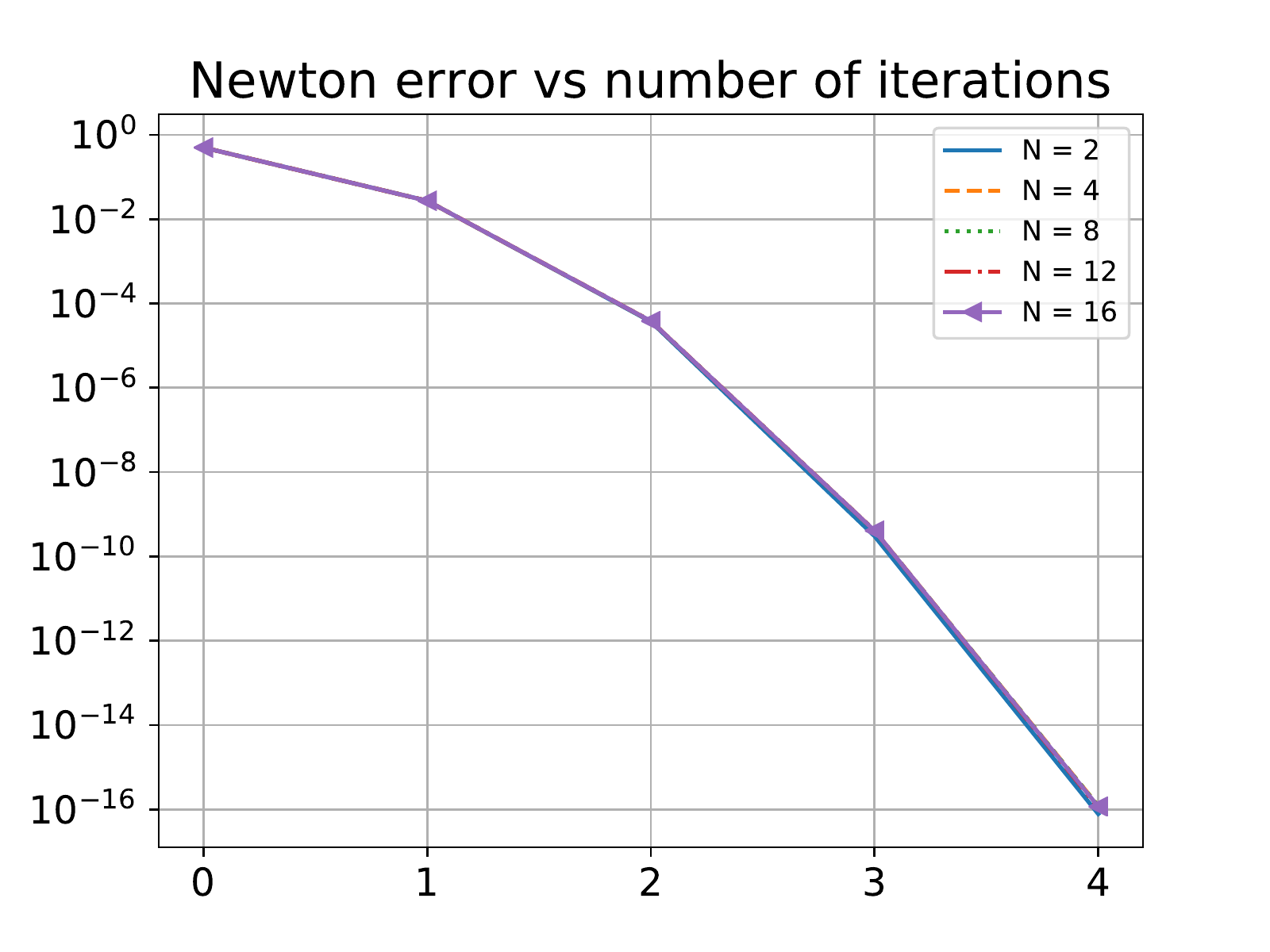}
\caption{Results of the first experiment for approximating the periodic solution to \eqref{eq:dae2} with $F$ defined by \eqref{eq:3dtoy}.  Here the initial guess in the Newton iteration to solve $R_N(\cdot,T)$ ($T=1$) is always set to be the zero vector.  (Left) Convergence plot of the $L^2$ residual error $E(N)$ vs the number of harmonics $N$.  (Right) Convergence plot of the Newton iteration showing the maximum of the absolute and relative error at each Newton iterate.  Note the nearly overlapping error in the Newton solves for all the tested values of $N$.}
\label{fig:3dtoy.1}
\end{figure}

\begin{figure}
\includegraphics[scale=0.4]{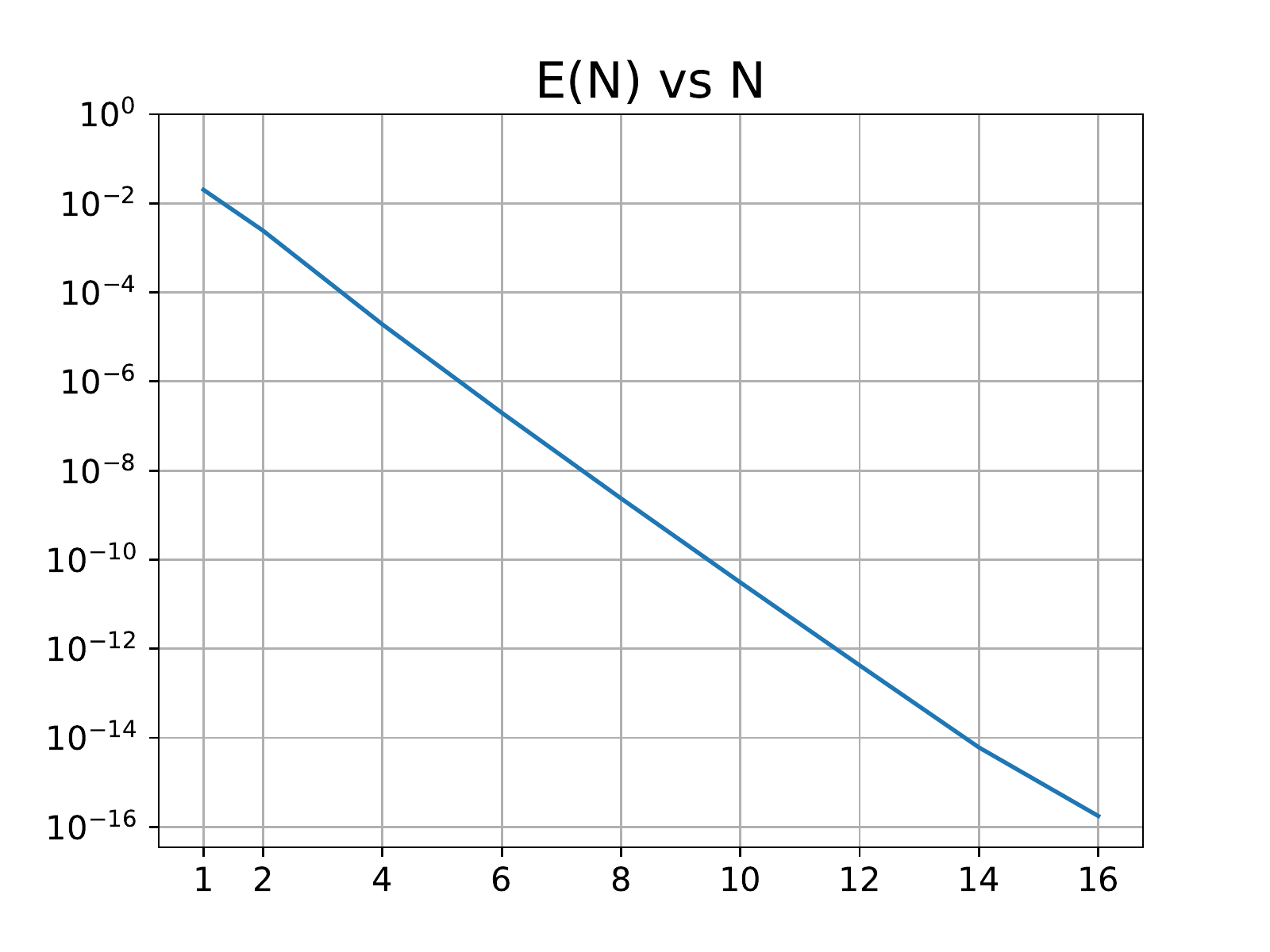}\includegraphics[scale=0.4]{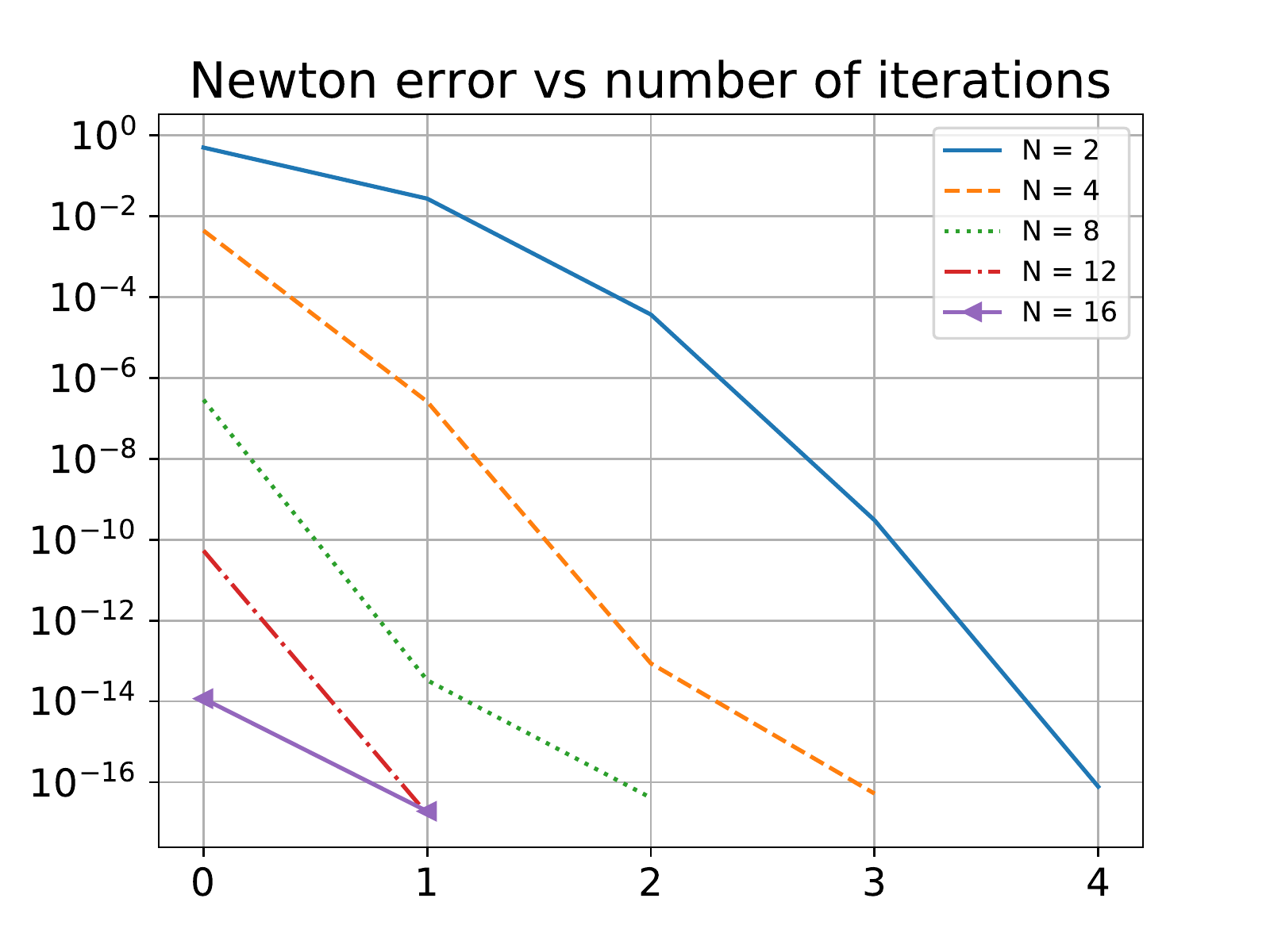}
\caption{Results of the second experiment for approximating the periodic solution to \eqref{eq:3dtoy}.  Here the initial guess in the Newton iteration to solve $R_N(\cdot,T)$ ($T=1$) is the zero vector when $N = 2$ and is given by the previously approximated solution to $R_{N-2}(\cdot,T)$ when $N > 2$.  (Left) Convergence plot of the $L^2$ residual error $E(N)$ vs the number of harmonics $N$.  (Right) Convergence plot of the Newton iteration showing the maximum of the absolute and relative error at each Newton iterate.}
\label{fig:3dtoy.2}
\end{figure}

\subsection{2D second order ODE}\label{sec:2dode}

Consider a model 2D mass-spring system (Equation 1.11 in \cite{ardehallen2014}) defined by a 2D, second order ODE with cubic nonlinearity. Let $k =2$, $X = \mb{R}^2$, and define $G \in C^{\infty}(X^{k+1},X)$ by:
\begin{equation}\label{eq:2dtoy}
G(u,u^{(1)},u^{(2)}) = \left[\begin{array}{c}
u_1^{(2)} \\
u_2^{(2)}
\end{array}\right] +  \underbrace{\left[\begin{array}{cc}2 & -1 \\ -1 & 2 \end{array} \right]}_{:= K}\left[\begin{array}{c}
 u_1 \\
 u_2
\end{array}\right]  + \left[\begin{array}{c}
u_1^3/2 \\
0
\end{array}\right] , \quad u = (u_1,u_2).
\end{equation}
The DAE \eqref{eq:dae2} with $G$ defined by \eqref{eq:2dtoy} and $g_1 = g_2 = I_X$ conserves the energy $H$ defined by:
\begin{equation}\label{eq:2dtoy.energy}
H(u,u^{(1)}) =  (u^{(1)})^T u^{(1)} / 2  + u^T K u / 2 +  u_1^4/8
\end{equation}
and has two branches of periodic solutions emanating from the origin.  Near the origin the periods of these two branches, Branch 1 and Branch 2, correspond respectively to the eigenvalues $\lambda_1 = -1$ and $\lambda_2 = -3$ of $K$.  Away from the origin the periods of the periodic orbits along these branches are unknown.  We therefore enforce a phase condition $\gamma$ defined as follows:
\begin{equation}\label{eq:2dtoy.phasecondition}
\gamma(q,q^{(1)})  = \int_0^1 [q(t)- Q_2(x_0)(t)]^T Q_2^{(1)}(x_0)(t) + [q^{(1)}(t) - Q_2^{(1)}(x_0)(t)]^T Q_2^{(2)}(x_0)(t) dt,
\end{equation}
where $x_0$ is the $X^2 = (\mb{R}^2)^2$ component of the initial guess for the zero of $\hat R_2$ (recall $\hat R_N = (R_N,\sigma_N)$ where $\sigma_N(x) = \gamma(Q_N(x),Q_N^{(1)}(x))$. Due to Theorem \ref{thm:unknown} and the fact that $F \in C^{\infty}(X^{k+1},X)$,  we expect that $E(N) = \mc{O}(e^{-\kappa N})$ for some $\kappa > 0$.

We construct an inexact Newton iteration to compute zeros of $\hat R_N$ where $N \geq 2$ as follows. The Fourier integrals defining $R_N$, the integral defining $\sigma$, and the integral defining $E(N)$ are approximated with an absolute and relative error tolerance of $10^{-15}$.  The Newton iterates are computed via direct linear solves using finite difference approximations for derivatives.   Relative and absolute tolerances of $10^{-15}$ are the stopping criteria used in the Newton iteration.

We set up an experiment to demonstrate convergence on Branches 1-2 with $H \approx 10$ as follows.  An initial guess formed from the eigenvalues and eigenvectors of $K$ is used to approximate the solution to $\hat R_2(x,\tau) = 0$ near the origin.  We then use pseudo-arclength continuation (see \cite[Chapter 4]{Keller1987}) to approximate a branch of solutions on which $\hat R_2 = 0$ until the the approximate periodic solution has $H \approx 10$.  Then, we successively compute zeros of $\hat R_4, \hat R_6,\hat R_8,\hdots,\hat R_{12}$ where the initial condition to compute $\hat R_N$ is the approximated zero of $\hat R_{N-2}$ when $N > 2$.

The results of our experiments are shown in Figures \ref{fig:2dtoy.1}-\ref{fig:2dtoy.2}.  Empirically, we find that the HB method achieves exponential convergence with $\kappa \approx 2.5$ and $\kappa \approx 3$ on Branch 1 and Branch 2, respectively.  The number of Newton iterations required to satisfy the error tolerances decreases with increasing $N$, similar to the second experiment in Section \ref{sec:3ddae}.

\begin{figure}
\includegraphics[scale=0.4]{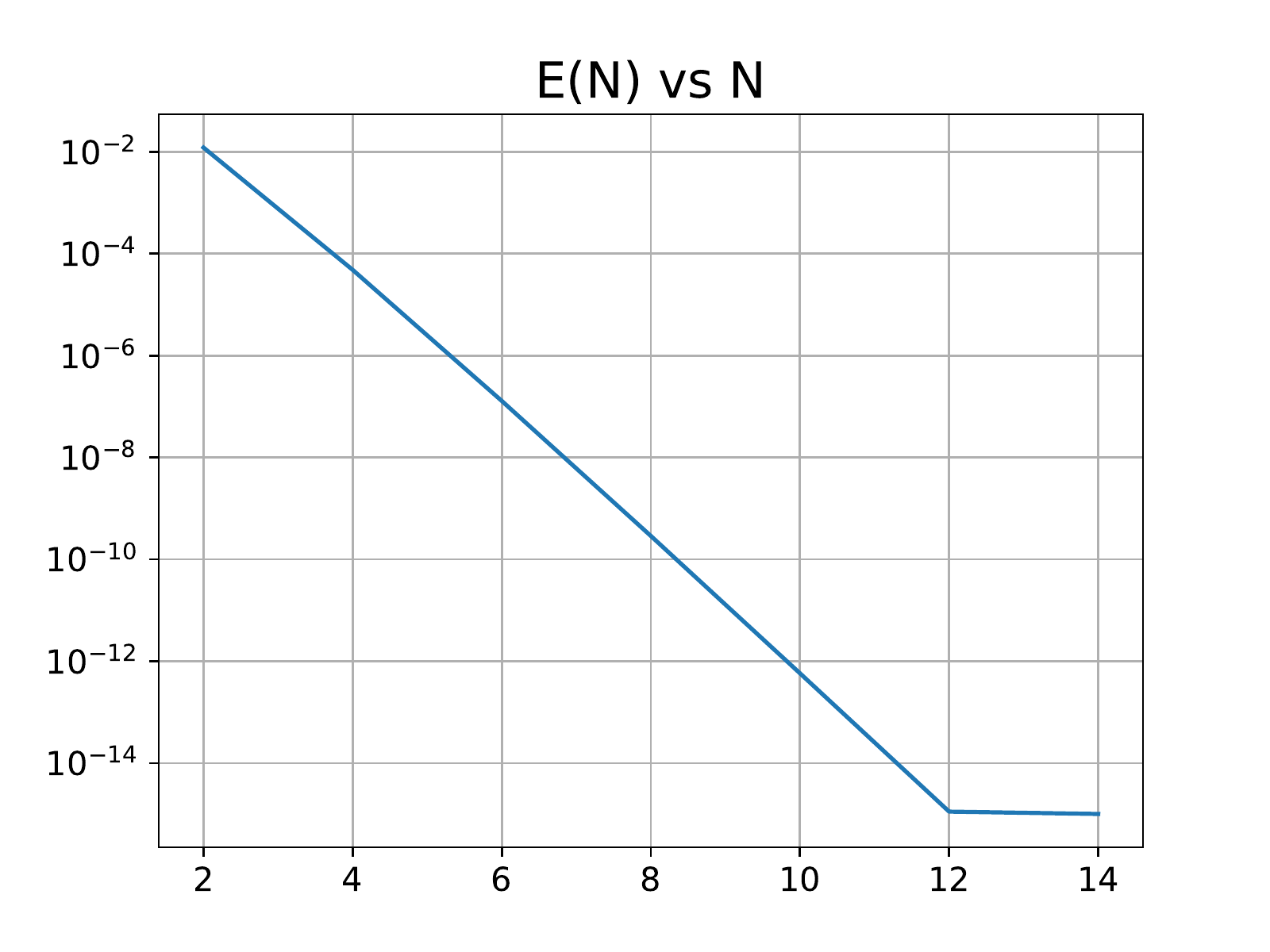}\includegraphics[scale=0.4]{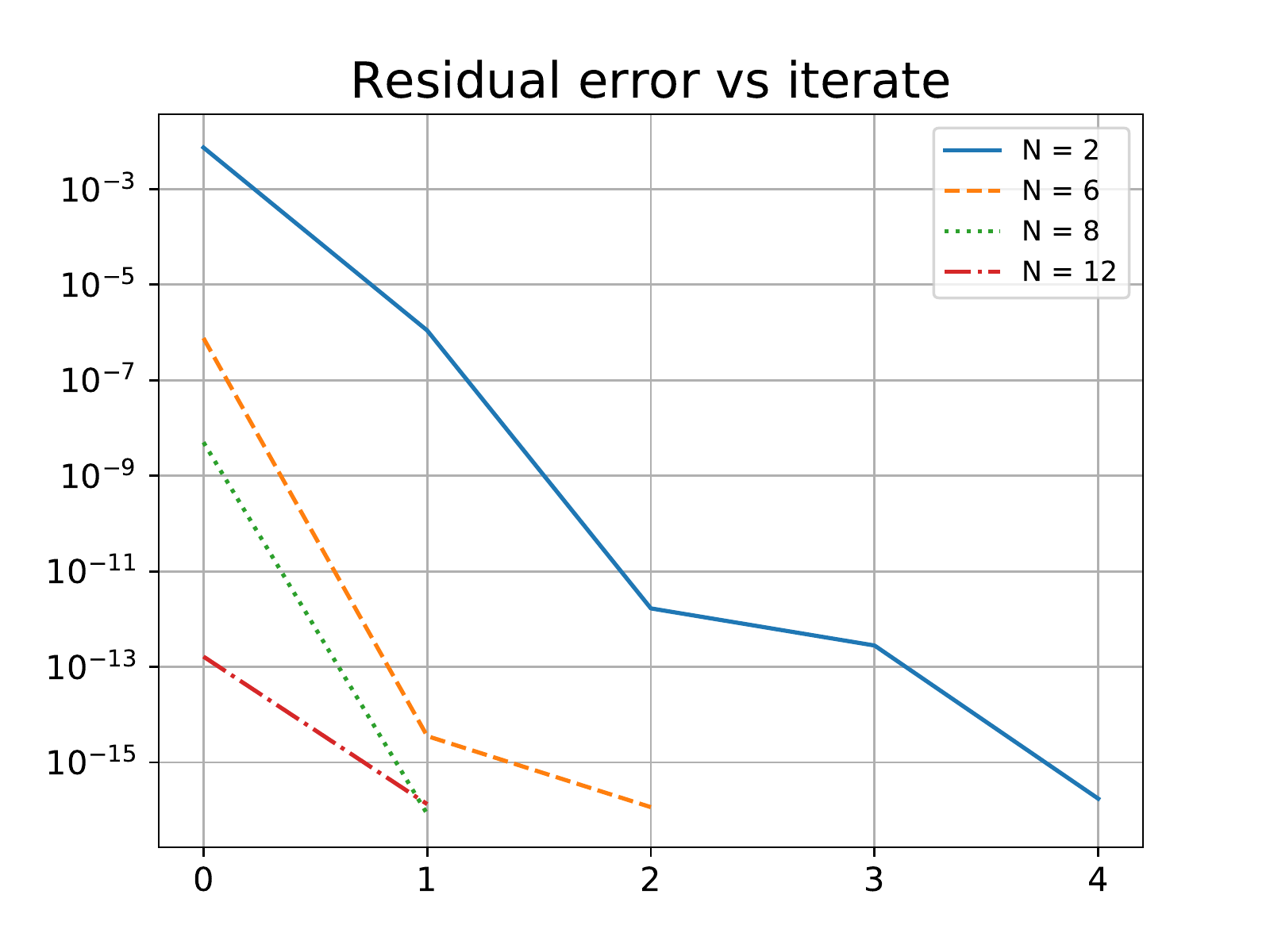}
\caption{Results of numerical experiment for approximating the periodic solution of \eqref{eq:dae2} with $F$ given by \eqref{eq:2dtoy} on Branch 1 with energy $H \approx 10$ .  (Left) Convergence plot of the $L^2$ residual error $E(N)$ vs the number of harmonics $N$.  (Right) Convergence plot of the Newton iteration showing the maximum of the absolute and relative residual error at each Newton iterate. 
}
\label{fig:2dtoy.1}
\end{figure}

\begin{figure}
\includegraphics[scale=0.4]{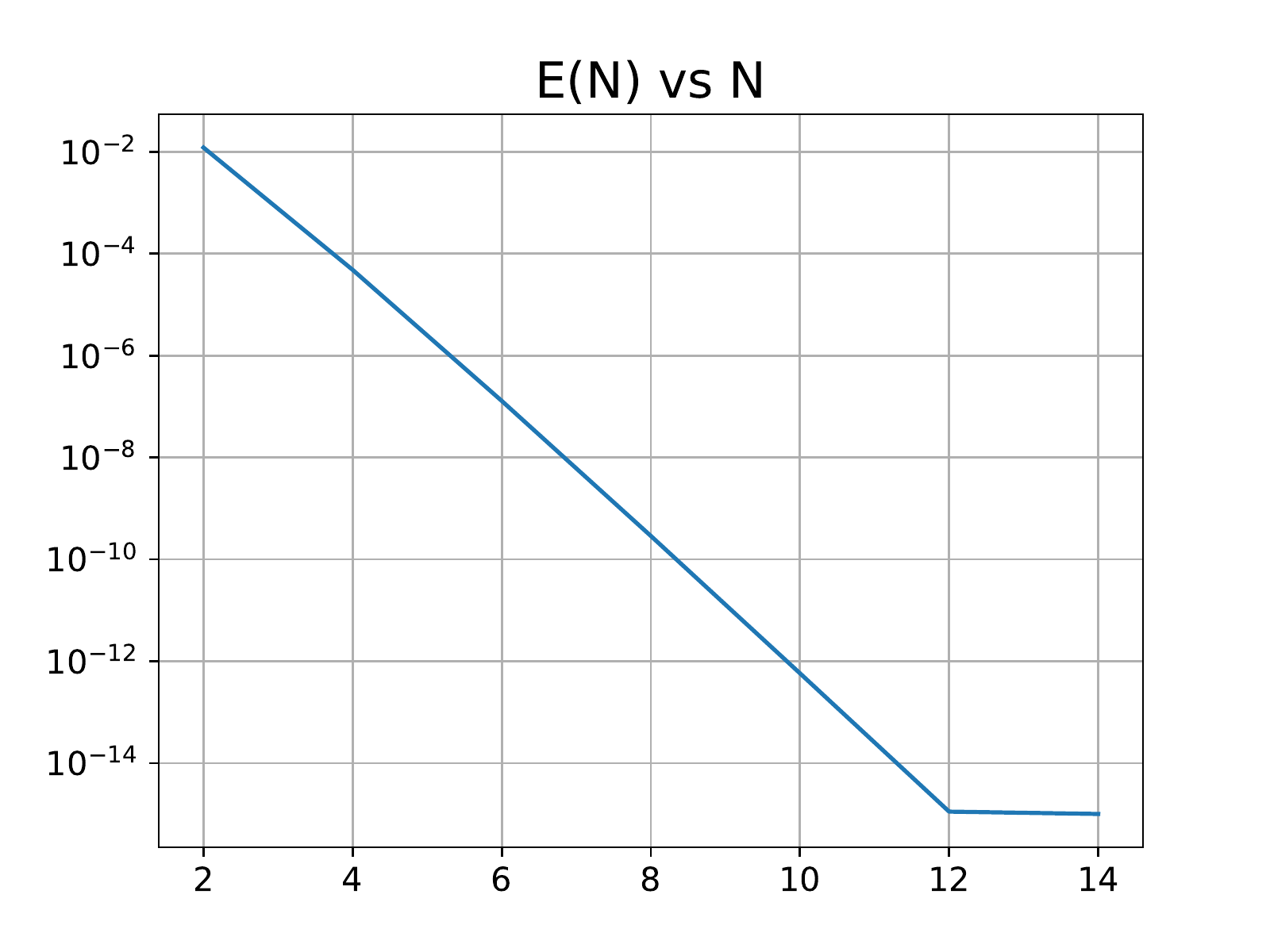}\includegraphics[scale=0.4]{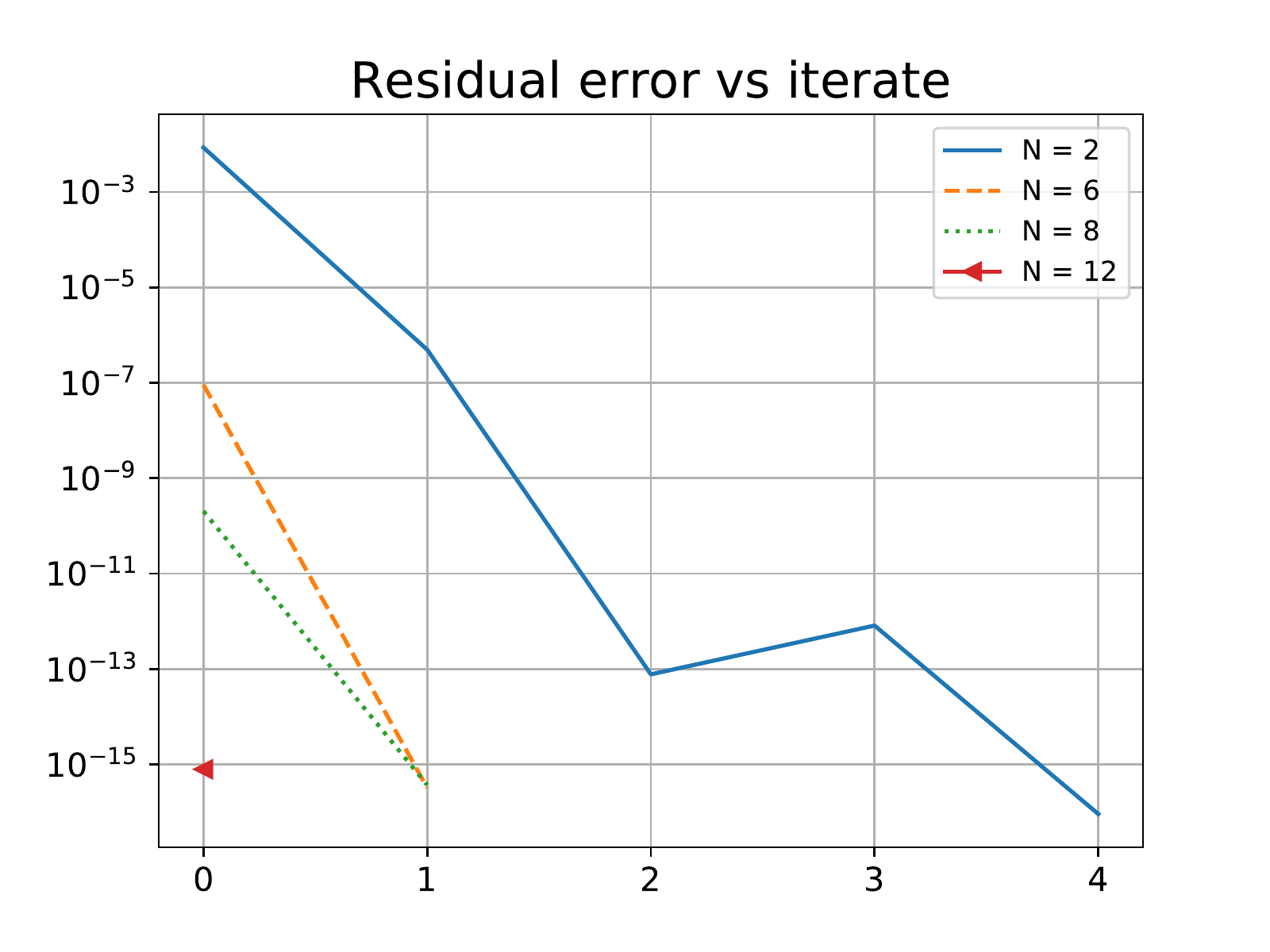}
\caption{Results of numerical experiment for approximating the periodic solution of \eqref{eq:dae2} with $F$ given by \eqref{eq:2dtoy} on Branch 2 with energy $H \approx 10$ .  (Left) Convergence plot of the $L^2$ residual error $E(N)$ vs the number of harmonics $N$.  (Right) Convergence plot of the Newton iteration showing the maximum of the absolute and relative residual error at each Newton iterate. Note the Newton iteration does note require any iterates when $N = 12$ since the initial guess satisfies the residual tolerances.}
\label{fig:2dtoy.2}
\end{figure}

\subsection{Cantilever beam with cubic spring}

In this example we consider a nonlinear mechanical system modeled as a cantilever beam with a cubic nonlinear spring at the beam tip.  The beam elements are modeled with structural steel material properties with a Young’s modulus of 205 GPa and a density of $7800$ kg/m$^3$.  We assume a constant square cross section with characteristic length of $1.4$ cm and a total length of 70 cm. At the beam tip, a grounded cubic spring with a stiffness constant of $6\cdot 10^9$ N/m$^3$ is attached to the end node in the transverse direction.   The beam is discretized into 19 linear 2D Euler-Bernoulli beam elements with two translational and one rotational degree of freedom at each node (see e.g. \cite[Chapter 2]{cooketa1998}), resulting in the following 57 dimensional DAE:
\begin{equation}\label{eq:cantileverbeam}
G(u,u^{(1)},u^{(2)},s t) = M u^{(2)} + C u^{(1)}  + K u + f_{nl}(u) - f(s t) , \quad u = u(t) \in \mb{R}^{57}.
\end{equation}
The matrices $M$, $C$, and $K$ correspond to the mass, viscous damping, and linear stiffness matrices, $f_{nl} \in C^{\infty}(\mb{R}^{57},\mb{R}^{57})$ is the cubic nonlinear spring applied to the degree of freedom corresponding to the transverse tip displacement, $f \in C^{\infty}(\mb{R},\mb{R}^{57})$ is a $2\pi$-periodic harmonic excitation function, and $s > 0$ is a parameter determining the forcing frequency.  We take $k=2$ and $X = \mb{R}^{57}$ with the standard Euclidean 2-norm.  Since $G \in C^{\infty}(X^{k+1} \times \mb{R},X)$ and due to Theorem \ref{thm:known}, we expect that $E(N) = \mc{O}(e^{-\kappa N})$ for some $\kappa > 0$. 

We describe the inexact Newton method we use to compute zeros of $R_N$.  Derivatives and Fourier integrals are formed using the alternating frequency/time method in \cite{CameronGriffin1989}.  An iterative GMRES linear solver with a relative residual tolerance of $\theta_N = 10^{-6}$ is used for the linear solve required at each Newton iterate.  The stopping criteria for the Newton iteration were (i) the absolute residual error is less than $10^{-10}$ and (ii) the Newton update ($\delta_j$ in Equation \eqref{eq:inexactNK.1} in Theorem \ref{thm:inexactNK}) has norm at most $10^{-6}$.  

In our study, for each $N$ we use pseudo-arclength continuation (see \cite[Chapter 4]{Keller1987}) to construct a branch of periodic solutions to \eqref{eq:cantileverbeam} using $s$ as the continuation parameter.  We let $E(N,s)$ denote the value of $E(N)$ for a given value of $s$  and we let $EB(N)$ be the maximal value of $E(N,\cdot)$ over the values of $s$ used in the continuation loop.  The continuation loop initializes with a zero vector initial guess.  

The results of our experiments are shown in Figure \ref{fig:cantileverbeam}.  Note that $E(N,\cdot)$ is not a function of $s$ since there can be two distinct solutions to \eqref{eq:cantileverbeam} with the same period.  The HB method empirically achieves exponential convergence $EB(N) = \mc{O}(e^{-\kappa_B N})$ with $\kappa_B \approx 1.2$.  We note that we obtain convergence of $EB(N)$ down to $10^{-9}$ which is consistent with the absolute tolerance enforced in the Newton iteration. We first remark that asymptotic convergence rate first becomes apparent when $5 \leq N \leq 9$.  The behavior of the error term $EB(N)$ varies widely for $N \leq 9$ and is especially variable for $N \leq 5$.   Second, we remark that although the tendency is that $EB(N)$ rapidly decays to zero with increasing $N$, the values of $E(N,\cdot)$ can vary by orders of magnitude along the branch of approximate solutions.

\begin{figure}
\includegraphics[scale=0.39]{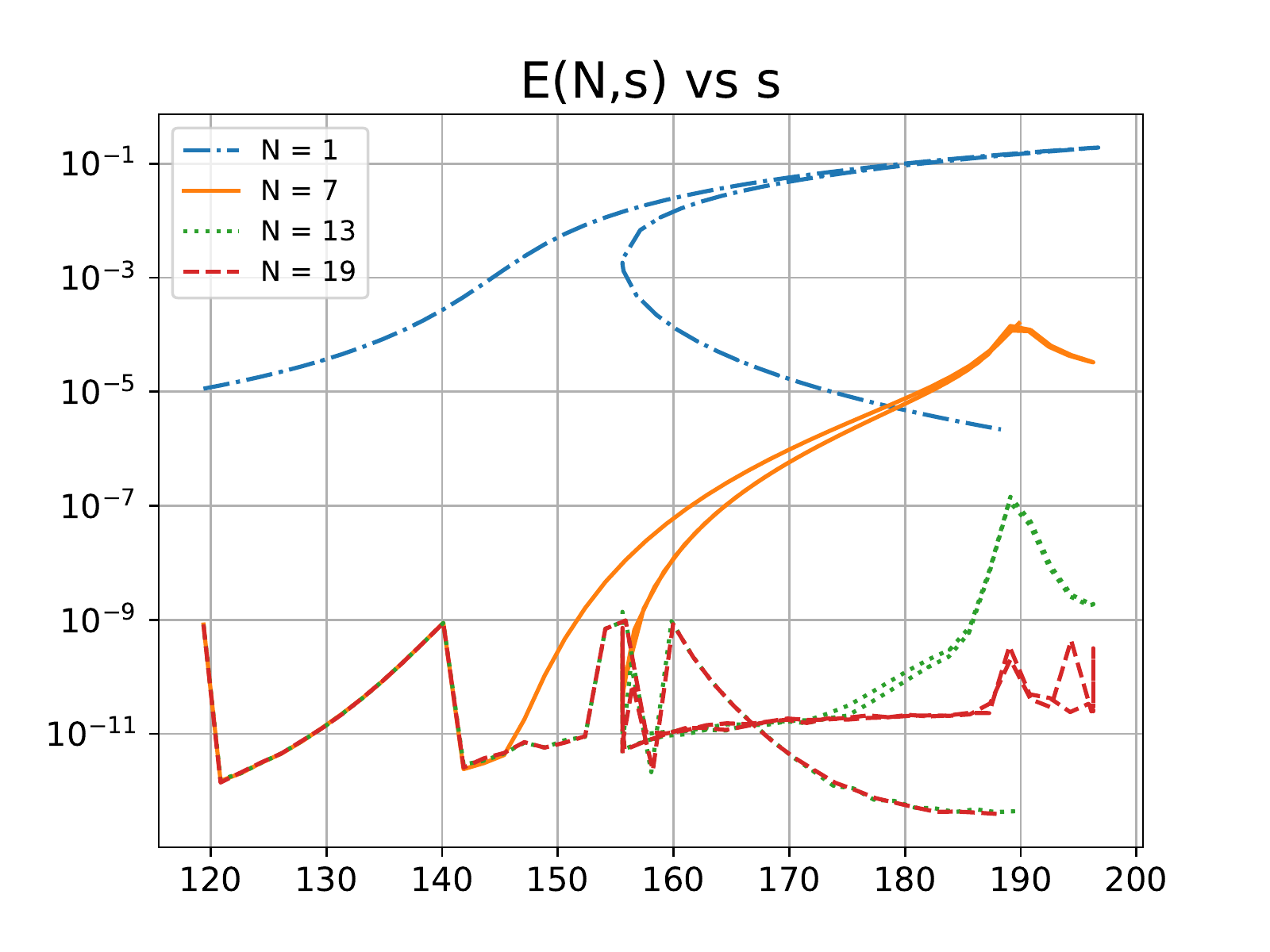}
\includegraphics[scale=0.39]{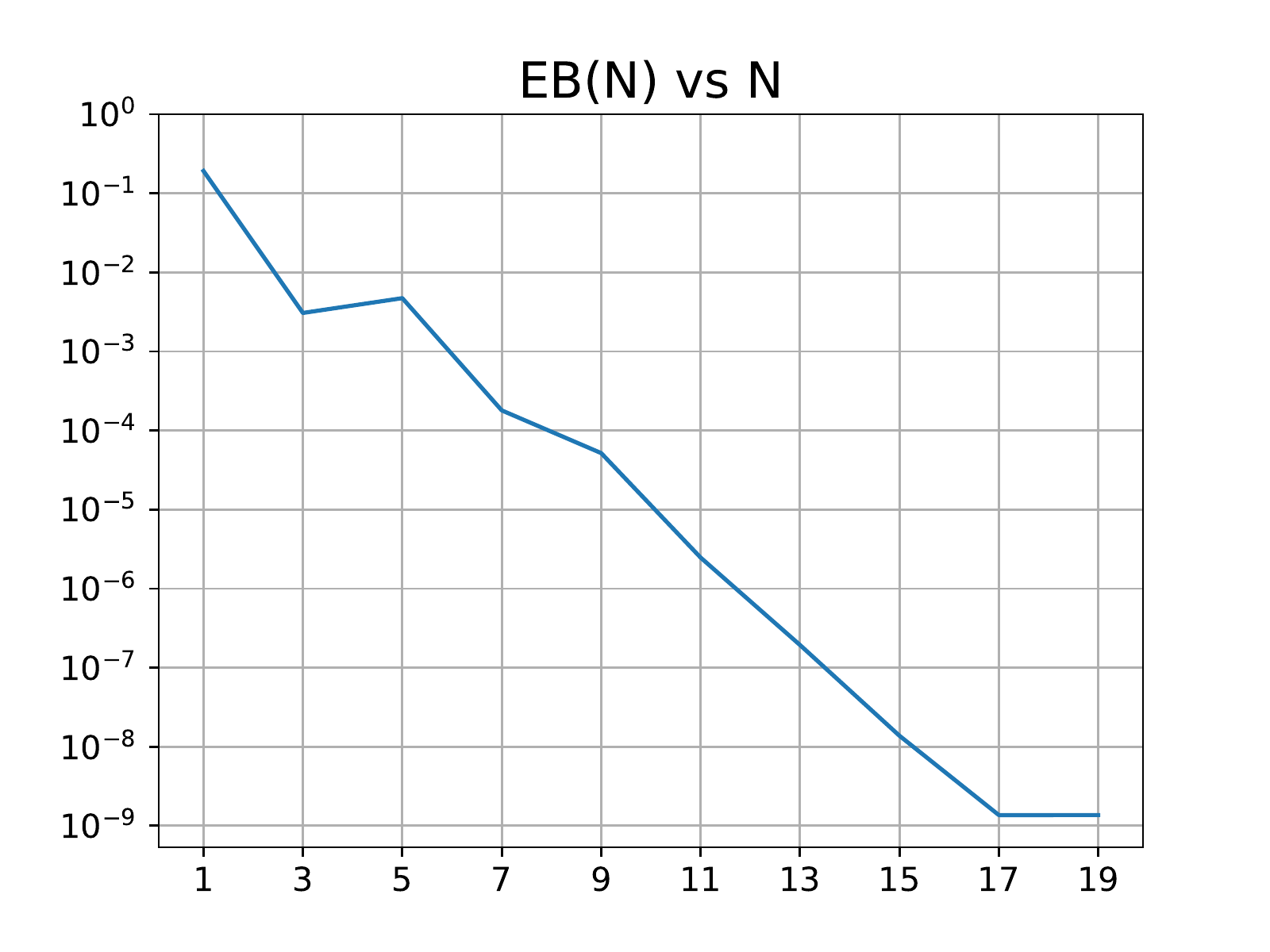}
\caption{Results of the numerical experiments for computing branches of periodic solutions to \eqref{eq:dae2} with $F$ defined by \eqref{eq:cantileverbeam}. (Left) The $L^2$ residual error $E(N,s)$ vs the continuation parameter $s$ for various values of $N$.  (Right) Convergence plot showing $EB(N)$ (the maximum of $E(N,s)$ for all values of $s$ in the continuation loop) vs the number of harmonics $N$.  }
\label{fig:cantileverbeam}
\end{figure}


\section{Conclusion}\label{sec:conclusion}

In this paper we have analyzed the convergence of the harmonic balance method for Hilbert space valued DAEs.  We have proved that, under mild smoothness and isolatedness assumptions, the approximate solution computed with the harmonic balance method converges to the exact solution with a rate determined by the rate of convergence of the Fourier series of the exact solution and the structure of the DAE.  Additionally, we have shown that an inexact Newton method used to compute the approximate solution determined by the harmonic balance method converges with the expected rate of convergence.  These theoretical results are empirically verified with several examples from structural mechanics and circuit modeling.  In future work we plan to focus on relaxing the smoothness hypotheses required in our theory and studying additional aspects such as stability of the harmonic balance method with respect to various types of perturbations.



\bibliographystyle{siamplain}
\bibliography{references}
\end{document}